\documentclass[10pt]{amsart}
\usepackage{amssymb,amsmath,txfonts,amsthm}
\usepackage{hyperref}
\usepackage{mathrsfs}
\newcommand\blfootnote[1]{%
  \begingroup
  \renewcommand\thefootnote{}\footnote{#1}%
  \addtocounter{footnote}{-1}%
  \endgroup
}
\newtheorem{theorem}{Theorem}
\newtheorem{prop}{Proposition}
\newtheorem{lemma}{Lemma}
\newtheorem*{remark}{Remark}
\newtheorem{claim}{Claim}
\newtheorem*{definition}{Definition}
\newtheorem{cor}{Corollary}
\newtheorem{conj}{Conjecture}

\def\Xint#1{\mathchoice
  {\XXint\displaystyle\textstyle{#1}}%
  {\XXint\textstyle\scriptstyle{#1}}%
  {\XXint\scriptstyle\scriptscriptstyle{#1}}%
  {\XXint\scriptscriptstyle\scriptscriptstyle{#1}}%
  \!\int}
\def\XXint#1#2#3{{\setbox0=\hbox{$#1{#2#3}{\int}$}
  \vcenter{\hbox{$#2#3$}}\kern-.5\wd0}}

\def\dashint{\Xint-}

\author{Gang Liu}
\address{Department of Mathematics\\University of California, Berkeley\\Berkeley, CA 94720}
\email{gangliu@math.berkeley.edu}
\title[Finite generation conjecture]{Gromov-Hausdorff limit of K\"ahler manifolds and the finite generation conjecture}
\date{}
\begin{document}
\begin{abstract}
We study the uniformization conjecture of Yau by using the Gromov-Haudorff convergence. As a consequence, we confirm Yau's finite generation conjecture. More precisely, 
on a complete noncompact K\"ahler manifold with nonnegative bisectional curvature, the ring of polynomial growth holomorphic functions is finitely generated.
During the course of the proof, we prove if $M^n$ is a complete noncompact K\"ahler manifold with nonnegative bisectional curvature and maximal volume growth, then $M$ is biholomorphic to an affine algebraic variety. We also confirm a conjecture of Ni on the existence of polynomial growth holomorphic functions on K\"ahler manifolds with nonnegative bisectional curvature.
\end{abstract}
\maketitle

\section{\bf{Introduction}}
\blfootnote{The author was partially supported by NSF grant DMS 1406593.}
In \cite{[Y1]}, Yau proposed to study the uniformization of complete K\"ahler manifolds
with nonnegative curvature. In particular, one wishes to determine whether or not a complete noncompact K\"ahler manifold with positive bisectional curvature is biholomorphic to a complex Euclidean space.
For this sake, Yau further asked in \cite{[Y1]}(see also page $117$ in \cite{[Y2]}) whether or not the ring of polynomial growth holomorphic functions is finitely generated, and whether or not dimension of the
spaces of holomorphic functions of polynomial growth is bounded from above by the
dimension of the corresponding spaces of polynomials on $\mathbb{C}^n$. 
Let us summarize Yau's questions in the three conjectures below:
\begin{conj}\label{conj1}
Let $M^n$ be a complete noncompact K\"ahler manifold with positive bisectional curvature. Then $M$ is biholomorphic to $\mathbb{C}^n$.
\end{conj}
\begin{conj}\label{conj2}
Let $M^n$ be a complete noncompact K\"ahler manifold with nonnegative bisectional curvature. Then the ring $\mathcal{O}_P(M)$ is finitely generated.\end{conj}
\begin{conj}\label{conj3}
Let $M^n$ be a complete noncompact K\"ahler manifold with nonnegative bisectional curvature. Then given any $d>0$, $dim(\mathcal{O}_d(M))\leq dim(\mathcal{O}_d(\mathbb{C}^n))$. \end{conj}
On a complete K\"ahler manifold $M$,  we say a holomorphic function $f\in \mathcal{O}_d(M)$, if there exists some $C>0$ with $|f(x)|\leq C(1+d(x, x_0))^d$ for all $x\in M$. Here $x_0$ is a fixed point on $M$. Let $\mathcal{O}_P(M) = \cup_{d>0}\mathcal{O}_d(M)$. 

Conjecture \ref{conj1} is open so far. However, there have been many important progresses due to various authors. In earlier works, Mok-Siu-Yau \cite{[MSY]} and Mok \cite{[Mo]} considered embedding by using holomorphic functions of polynomial growth. Later, with K\"ahler-Ricci flow, results were improved significantly \cite{[Shi1]}\cite{[Shi2]}\cite{[CZ]}\cite{[N2]}\cite{[CT]}.

Conjecture \ref{conj3} was confirmed by Ni \cite{[N1]} with the assumption that $M$ has maximal volume growth. Later, by using Ni's method, Chen-Fu-Le-Zhu \cite{[CFLZ]} removed the extra condition.  See also \cite{[L1]} for a different proof. The key of Ni's method is a monotonicity formula for heat flow on K\"ahler manifold with nonnegative bisectional curvature. 

Despite great progresses of conjecture \ref{conj1} and conjecture \ref{conj3}, not much is known about conjecture \ref{conj2}.
In \cite{[Mo]}, Mok proved the following:
\begin{theorem}[Mok] 
Let $M^n$ be a complete noncompact K\"ahler manifold with positive bisectional curvature  such that for some fixed point $p\in M$, 
\begin{itemize}
\item Scalar curvature $\leq \frac{C_0}{d(p, x)^2}$ for some $C_0>0$;
\item $Vol(B(p, r))\geq C_1r^{2n}$ for some $C_1>0$.
\end{itemize}
Then $M^n$ is biholomorphic to an affine algebraic variety. 
\end{theorem}

In Mok's proof, the biholomorphism was given by holomorphic functions of polynomial growth. Therefore, $\mathcal{O}_P(M)$ is finitely generated. 
In the general case, it was proved by Ni \cite{[N1]} that the transcendental dimension of $\mathcal{O}_P(M)$ over $\mathbb{C}$ is at most $n$. However, this does not imply the finite generation of $\mathcal{O}_P(M)$.
The main result in this paper is the confirmation of conjecture \ref{conj2} in the general case:
\begin{theorem}\label{thm1}
Let $M^n$ be a complete noncompact K\"ahler manifold with nonnegative bisectional curvature. Then the ring $\mathcal{O}_P(M)$ is finitely generated.
\end{theorem}
During the course of the proof, we obtain a partial result for conjecture \ref{conj1}:
\begin{theorem}\label{thm2}
Let $M^n$ be a complete noncompact K\"ahler manifold with nonnegative bisectional curvature. Assume $M$ is of maximal volume growth, then $M$ is biholomorphic to an affine algebraic variety.
\end{theorem}
Here maximal volume growth means $Vol(B(p, r))\geq Cr^{2n}$ for some $C>0$.
This seems to be the first uniformization type result without assuming the curvature upper bound.

If one wishes to prove conjecture \ref{conj1} by considering $\mathcal{O}_P(M)$, it is important to know when $\mathcal{O}_P(M)\neq \mathbb{C}$.
In \cite{[N1]}, Ni proposed the following interesting conjecture:
\begin{conj}\label{conj4}
Let $M^n$ be a complete noncompact K\"ahler manifold with nonnegative bisectional curvature. Assume $M$ has positive bisectional curvature at one point $p$. Then the following three conditions are equivalent:

(1) $\mathcal{O}_P(M)\neq \mathbb{C}$;

(2)$M$ has maximal volume growth;

(3)  there exists a constant $C$ independent of $r$ so that $\dashint_{B(p, r)} S \leq \frac{C}{r^2}$. Here $S$ is the scalar curvature. $\dashint$ means the average.

\end{conj}
In complex one dimensional case, the conjecture is known to hold, e.g., \cite{[LT2]}. For higher dimensions, 
Ni proved $(1)$ implies $(3)$ in \cite{[N1]}. The proof used the heat flow method. Then in \cite{[NT2]}, Ni and Tam proved that $(3)$ also implies $(1)$. Their proof employs the Poincare-Lelong equation and the heat flow method. Thus, it remains to prove $(1)$ and $(2)$ are equivalent. Under some extra conditions, Ni \cite{[N2]} and Ni-Tam \cite{[NT1]} were able to prove the equivalence of $(1)$ and $(2)$. In \cite{[L2]}, the author proved that $(1)$ implies $(2)$. In fact, the condition that $M$ has positive bisectional curvature at one point could be relaxed to that the universal cover of $M$ is not a product of two K\"ahler manifolds. 

In this paper, we prove that $(2)$ also implies $(1)$. Thus conjecture \ref{conj4} is solved in full generality.
More precisely, we prove
\begin{theorem}\label{thm3}
Let $(M^n, g)$ be a complete K\"ahler manifold with nonnegative bisectional curvature and maximal volume growth. Then there exists a nonconstant holomorphic function with polynomial growth on $M$.
\end{theorem}
This theorem might be compared with the following open question in Riemannian geometry:
\begin{conj}
Let $M^n$ be a complete noncompact Riemannian manifold with nonnegative Ricci curvature and maximal volume growth. Then there exists a nonconstant polynomial growth harmonic function.
\end{conj}

The strategy of the proofs in this paper is very different from earlier works. Here we make use of several different techniques: the Gromov-Hausdorff convergence theory developed by Cheeger-Colding \cite{[CC1]}\cite{[CC2]}\cite{[CC3]}\cite{[CC4]}, Cheeger-Colding-Tian \cite{[CCT]}; the heat flow method by Ni \cite{[N1]} and Ni-Tam \cite{[NT1]}\cite{[NT2]}; the Hormander $L^2$-technique \cite{[Ho]}\cite{[D]} ; the three circle theorem \cite{[L1]}.

The first key point is to prove theorem \ref{thm3}. By Hormander's $L^2$-technique, to produce holomorphic functions of polynomial growth, it suffices to construct strictly plurisubharmonic function of logarithmic growth. However, it is not obvious how to construct such function by only assuming the maximal volume growth condition. In \cite{[MSY]}\cite{[Mo]}, Mok-Siu-Yau and Mok considered the Poincare-Lelong equation $\sqrt{-1}\partial\overline\partial u = Ric$. When the curvature has pointwise quadratic decay, they were able to prove the existence of a solution with logarithmic growth. Later, Ni and Tam \cite{[NT1]}\cite{[NT2]} were able to relax the condition to that the curvature has average quadratic decay. Then it suffices to prove that maximal volume growth implies the average curvature decay. 

We prove theorem \ref{thm3} by a different strategy. We first blow down the manifold. Then by using the Cheeger-Colding theory, heat flow technique and Hormander $L^2$ theory, we construct holomorphic functions with controlled growth in a sequence of exhaustion domains on $M$. Then three circle theorem ensures that we can take subsequence to obtain a nonconstant holomorphic function with polynomial growth.

Once theorem \ref{thm3} is proved, Hormander's $L^2$ technique produces a lot of holomorphic functions of polynomial growth. It turns out $\mathcal{O}_P(M)$ separates points and tangent spaces on $M$. However, since the manifold is not compact, it does not follow directly that $M$ is affine algebraic. In Mok's paper \cite{[Mo]}, the proof of this part took more than $35$ pages, even with the additional assumption that curvature has pointwise quadratic decay.

In our case, there is a serious difficulty to prove that the map given by $\mathcal{O}_d(M)$ is proper.
We overcome this difficulty in theorem \ref{thm4}. Again, the idea is new. We will apply the induction on the dimension of splitting factor for a tangent cone. All techniques above are employed. 

Once we prove the properness of the holomorphic map, it is straightforward to prove $M$ is affine algebraic by using techniques from complex analytic geometry. Here the argument resembles some part in \cite{[DS]}.
Then we conclude conjecture \ref{conj2} when the manifold has maximal volume growth.
For the general case, we apply the main result in \cite{[L2]}. It suffices to handle the case when the universal cover of the manifold splits. Then we need to consider group actions. The final result follows from an algebraic result of Nagata \cite{[Na]}.

This paper is organized as follows. In section $2$, we collect some preliminary results necessary for this paper. In section $3$, we prove a result which controls the size of a holomorpic chart when the manifold is Gromov-Hausdorff close to an Euclidean ball.  As the first application, we prove in section $4$ a gap theorem for the complex structure of $\mathbb{C}^n$. Section $5$ deals with the proof of theorem \ref{thm3}. The proof of theorem \ref{thm4} is contained in section $6$. Finally, the proof of theorem \ref{thm1} is given section $7$.

There are two appendices. For appendix $A$, we present a result of Ni-Tam in \cite{[NT1]} which was not stated explicitly (here we are not claiming any credits).
In appendix $B$, we introduce some results of Nagata \cite{[Na]} to conclude the proof of the main theorem.

\bigskip

\begin{center}
\bf  {\quad Acknowledgment}
\end{center}
The author would like to express his deep gratitude to Professors John Lott, Lei Ni, Jiaping Wang for many valuable discussions during the work. He thanks Professor Xinyi Yuan for some discussions on algebra. He also thanks Professors Peter Li, Luen-Fai Tam and Shing-Tung Yau for the interest in this work.

\section{\bf{Preliminary results}}

First recall some convergence results for manifolds with Ricci curvature lower bound. 
Let $(M^n_i, y_i, \rho_i)$ be a sequence of pointed complete Riemannian manifolds, where $y_i\in M^n_i$ and $\rho_i$ is the metric on $M^n_i$. By Gromov's compactness theorem, if $(M^n_i, y_i, \rho_i)$ have a uniform lower bound of the Ricci curvature, then a subsequence converges to some $(M_\infty, y_\infty, \rho_\infty)$ in the Gromov-Hausdorff topology. See \cite{[G]} for the definition and basic properties of Gromov-Hausdorff convergence.
\begin{definition}
Let $K_i\subset M^n_i\to K_\infty\subset M_\infty$ in the Gromov-Hausdorff topology. Assume $\{f_i\}_{i=1}^\infty$ are functions on $M^n_i$, $f_\infty$ is a function on $M_\infty$.  
$\Phi_i$ are $\epsilon_i$-Gromov-Hausdorff approximations, $\lim\limits_{i\to\infty} \epsilon_i = 0$. If $f_i\circ \Phi_i$ converges to $f_\infty$ uniformly, we say $f_i\to f_\infty$ uniformly over $K_i\to K_\infty$.
\end{definition}
 In many applications, $f_i$ are equicontinuous. The Arzela-Ascoli theorem applies to the case when the spaces are different.  When $(M_i^n, y_i, \rho_i)\to (M_\infty, y_\infty, \rho_\infty)$ in the Gromov-Hausdorff topology, any bounded, equicontinuous sequence of functions $f_i$ has a subsequence converging uniformly to some $f_\infty$ on $M_\infty$.

Let the complete pointed metric space $(M_\infty^m, y)$ be the Gromov-Hausdorff limit of a sequence of connected pointed Riemannian manifolds, $\{(M_i^n, p_i)\}$, with $Ric(M_i)\geq 0$. Here $M_\infty^m$ has Haudorff dimension $m$ with $m\leq n$. A tangent cone at $y\in M_\infty^m$ is a complete pointed Gromov-Hausdorff limit $((M_\infty)_y, d_\infty, y_\infty)$ of $\{(M_\infty, r_i^{-1}d, y)\}$, where $d, d_\infty$ are the metrics of $M_\infty, (M_\infty)_y$ respectively, $\{r_i\}$ is a positive sequence converging to $0$.

\begin{definition}
A point $y\in M_\infty$ is called regular, if there exists some $k$ so that every tangent cone at $y$ is isometric to $\mathbb{R}^k$. A point is called singular, if it is not regular.
\end{definition}
In \cite{[CC2]}, the following theorem was proved: 
\begin{theorem}\label{thm5}
Regular points are dense in the Gromov-Haudorff limit of manifolds with Ricci curvature lower bound.
\end{theorem}

\medskip

Results of heat flow on K\"ahler manifolds by Ni-Tam \cite{[NT1]}:
\begin{theorem}\label{thm6}
Let $M^n$ be a complete noncompact K\"ahler manifold with nonnegative bisectional curvature. Let $u$ be a smooth function on $M$ with compact support. Let $$v(x, t) =\int_MH(x, y, t)u(y)dy.$$ Here $H(x, y, t)$ is the heat kernel of $M$. Let $\eta(x, t)_{\alpha\overline\beta} = v_{\alpha\overline\beta}$ and $\lambda(x)$ be the minimum eigenvalue for $\eta(x, 0)$. Let $$\lambda(x, t) = \int_M H(x, y, t)\lambda(y)dy.$$ Then $\eta(x, t)-\lambda(x, t)g_{\alpha\overline{\beta}}$ is a nonnegative $(1, 1)$ tensor for $t\in [0, T]$ for $T>0$.
\end{theorem}
A detailed proof of this theorem is presented in appendix $A$.

\medskip

Hormander $L^2$ theory:
\begin{theorem}\label{thm7}
Let $(X, \omega)$ be a connected but not necessarily complete K\"ahler manifold with $Ric\geq 0$. Assume $X$ is Stein. Let $\varphi$ be a $C^\infty$ function on $X$ with $\sqrt{-1}\partial\overline\partial\varphi \geq c\omega$ for some positive function $c$ on $X$. Let $g$ be a smooth $(0, 1)$ form satisfying $\overline\partial g = 0$ and $\int_X\frac{|g|^2}{c}e^{-\varphi}\omega^n<+\infty$, then there exists a smooth function $f$ on $X$ with $\overline\partial f = g$ and $\int_X |f|^2e^{-\varphi}\omega^n\leq \int_X\frac{|g|^2}{c}e^{-\varphi}\omega^n$.\end{theorem}
The proof can be found in \cite{[D]}, page 38-39. Also compare lemma 4.4.1 in \cite{[Ho]}.

\medskip
Three circle theorem in \cite{[L1]}:
\begin{theorem}\label{thm8}
Let $M$ be a complete noncompact K\"ahler manifold with nonnegative holomorphic sectional curvature, $p\in M$. Let $f$ be a holomorphic function on $M$. Let $M(r) = \max\limits_{B(p, r)}|f(x)|$. Then $\log M(r)$ is a convex function of $\log r$. Therefore, given any $k>1$, $\frac{M(kr)}{M(r)}$ is monotonic increasing.
\end{theorem}
This theorem has the following consequences:
\begin{cor}\label{cor1}
Given the same condition as in theorem \ref{thm8}. If $f\in\mathcal{O}_d(M)$, then $\frac{M(r)}{r^d}$ is non-increasing.
\end{cor}
\begin{cor}\label{cor2}
Given the same condition as in theorem \ref{thm8}. If $f(p) = 0$, then $\frac{M(r)}{r}$ is non-decreasing.
\end{cor}
\begin{remark}
The three circle theorem is still true for holomorphic sections on nonpositive bundles. See page $17$ of \cite{[L1]} for a proof.
\end{remark}

\medskip
A multiplicity estimate by Ni \cite{[N1]}(see also \cite{[CFLZ]}):
\begin{theorem}\label{thm9}
Let $M^n$ be a complete noncompact K\"ahler manifold with nonnegative bisectional curvature. Then $dim(\mathcal{O}_d(M))\leq dim(\mathcal{O}_d(\mathbb{C}^n))$.\end{theorem}
Note this result also follows from corollary \ref{cor1}.

\medskip

In this paper, we will denote by $\Phi(u_1,..., u_k|....)$ any nonnegative functions depending on $u_1,..., u_k$ and some additional parameters such that when these parameters are fixed, $$\lim\limits_{u_k\to 0}\cdot\cdot\cdot\lim\limits_{u_1\to 0}\Phi(u_1,..., u_k|...) = 0.$$  Let $C(n), C(n, v)$ be large positive constants depending only on $n$ or $n, v$; $c(n), c(n, v)$ be small positive constants depending only on $n$ or $n, v$. The values might change from line to line.

\section{\bf{Construct holomorphic charts with uniform size}}
In this section, we introduce the following proposition which is crucial for the construction of holomorphic functions.
\begin{prop}\label{prop1}
Let $M^n$ be a complete K\"ahler manifold with nonnegative bisectional curvature, $x\in M$. There exist $\epsilon(n)>0, \delta = \delta(n)>0$ so that the following holds: For $\epsilon<\epsilon(n)$, if $d_{GH}(B(x, \frac{1}{\epsilon}r), B_{\mathbb{C}^n}(0, \frac{1}{\epsilon}r))<\epsilon r$, there exists a holomorphic chart $(w_1,...., w_n)$ containing $B(x, \delta r)$ so that 
\begin{itemize}
\item $w_s(x) = 0(1\leq s\leq n)$.
\item $|\sum\limits_{s=1}^n|w_s|^2(y)-r^2(y)|\leq\Phi(\epsilon |n)r^2$ in $B(x, \delta r)$. Here $r(y) = d(x, y)$. \item $|dw_s(y)|\leq C(n)$ in $B(x, \delta r)$.\end{itemize} 
\end{prop}
\begin{proof} 
By scaling, we may assume $r>>1$ which is to be determined.  Set $R = \frac{r}{100}>>1$.
According to the assumptions and the Cheeger-Colding theory \cite{[CC1]}(also equation ($1.23$) in \cite{[CC3]}), 
there exist real harmonic functions $b_1,..., b_{2n}$ in $B(x,4r)$ so that 
\begin{equation}\label{eq1}\dashint_{B(x, 2r)}  \sum\limits_{j}|\nabla(\nabla b_j)|^2 +\sum\limits_{j, l}|\langle\nabla b_j, \nabla b_l\rangle - \delta_{jl}|^2\leq \Phi(\epsilon|n, r)\end{equation} and \begin{equation}\label{eq2}b_j(x) = 0(1\leq j\leq 2n); |\nabla b_j|\leq C(n)\end{equation} in $B(x, 2r)$. Moreover, the map $F(y) = (b_1(y),..., b_{2n}(y))$ is a $\Phi(\epsilon|n)r$ Gromov-Hausdorff approximation from $B(x, 2r)$ to $B_{\mathbb{R}^{2n}}(0, 2r)$. 
According to the argument above lemma $9.14$ in \cite{[CCT]}(see also $(20)$ in \cite{[L2]}), after a suitable orthogonal transformation, we may assume \begin{equation}\label{eq3}\dashint_{B(x, r)} |J\nabla b_{2s-1} - \nabla b_{2s}|^2 \leq \Phi(\epsilon|n, r)\end{equation} for $1\leq s\leq n$. Set $w'_s = b_{2s-1}+\sqrt{-1}b_{2s}$. Then \begin{equation}\label{eq4}\dashint_{B(x, r)} |\overline\partial w'_s|^2\leq \Phi(\epsilon|n, r).\end{equation}
The idea is to perturb $w'_s$ so that they become a holomorphic chart. We would like to apply the Hormander $L^2$-estimate. First, we construct the weight function.
Consider the function $$h(y) = \sum\limits_{j=1}^{2n}b_j^2(y).$$ Then in $B(x, r)$, \begin{equation}\label{eq5}|h(y)-r^2(y)|\leq\Phi(\epsilon|n)r^2.\end{equation}
By (\ref{eq2}), \begin{equation}\label{eq6}|\nabla h(y)|\leq C(n)r(y)\end{equation} in $B(x, r)$. The real Hessian of $h$ satisfies \begin{equation}\label{eq7}\int_{B(x, 5R)}\sum\limits_{u, v}|h_{uv}(y)- 2g_{uv}|^2 \leq \Phi(\epsilon|n, R).\end{equation}
Now consider a smooth function $\varphi$: $\mathbb{R}^+\to \mathbb{R}^+$ with $\varphi(t) = t$ for $0\leq t\leq 1$; $\varphi(t) = 0$ for $t\geq 2$; $|\varphi|, |\varphi'|, |\varphi''|\leq C(n)$. Let $H(x, y, t)$ be the heat kernel on $M$ and set \begin{equation}\label{eq8}u(y) = 5R^2\varphi(\frac{h(y)}{5R^2}), u_t(z) = \int_M H(z, y, t)u(y)dy.\end{equation}
\begin{claim}\label{cl1}
$u_1(z)$ satisfies that $(u_1)_{\alpha\overline\beta}(z)\geq c(n)g_{\alpha\overline\beta}>0$ in $B(x, \frac{R}{10})$. 
\end{claim}
\begin{proof}
Let $\lambda(y)$ be the lowest eigenvalue of the complex hessian $u_{\alpha\overline\beta}$. 
By (\ref{eq7}), $$\dashint_{B(x, 5R)}|h_{\alpha\overline\beta} - 2g_{\alpha\overline\beta}|^2\leq \Phi(\epsilon|n, R).$$
Then there exists $E\subset B(x, 5R)$ with \begin{equation}\label{eq9}vol(B(x, 5R)\backslash E)\leq \Phi(\epsilon|n, R); h_{\alpha\overline\beta}\geq \frac{1}{2}g_{\alpha\overline\beta}\end{equation} on $E$. 
By (\ref{eq5}), we may assume $h(y)\leq 5R^2$ in $B(x, 2R)$.  Then $u = h$ in $B(x, 2R)$. We have \begin{equation}\label{eq10}\begin{aligned}(\int_{B(x, 2R)\backslash E}|\lambda^2(y)|dy)^{\frac{1}{2}}&\leq (\int_{B(x, 4R)\backslash E}\sum\limits_{\alpha, \beta}|h_{\alpha\overline\beta}|^2)^{\frac{1}{2}}\\&\leq4(\int_{B(x, 4R)\backslash E}\sum\limits_{\alpha, \beta}|h_{\alpha\overline\beta}-2g_{\alpha\overline\beta}|^2dy)^{\frac{1}{2}}+4(\int_{B(x, 4R)\backslash E}\sum\limits_{\alpha, \beta}|2g_{\alpha\overline\beta}|^2dy)^{\frac{1}{2}}\\&\leq \Phi(\epsilon|n, R).\end{aligned}\end{equation}  \begin{equation}\label{eq11}\begin{aligned}|\lambda|&\leq |u_{\alpha\overline\beta}| \\&= |\varphi' h_{\alpha\overline\beta}+\frac{\varphi''}{5R^2} h_{\alpha}h_{\overline\beta}| \\&\leq |\varphi' (h_{\alpha\overline\beta}-2g_{\alpha\overline\beta})|+|2\varphi' g_{\alpha\overline\beta}+\frac{\varphi''}{5R^2} h_{\alpha}h_{\overline\beta}|\\&\leq C(n)(|h_{\alpha\overline\beta}-2g_{\alpha\overline\beta}|+1).\end{aligned}\end{equation} Therefore, \begin{equation}\label{eq12}\int_{B(x, 5R)} |\lambda(y)|dy\leq C(n)R^{2n}.\end{equation}
Let $\lambda(z, 1) = \int H(z, y, 1)\lambda(y)dy$. Note by definition (\ref{eq8}), $u$ is supported in $B(x, 4R)$.
By (\ref{eq9}), $\lambda\geq\frac{1}{2}$ in $E$. For $z\in B(x, \frac{R}{10})$, \begin{equation}\label{eq13}\begin{aligned}\int H(z, y, 1)\lambda(y)dy &= \int_{B(x, 4R)} H(z, y, 1)\lambda(y)dy \\&\geq \int_{B(x, 2R)\backslash E} H(z, y, 1)\lambda(y)dy +\int_{B(x, 4R)\backslash B(x, 2R)} H(z, y, 1)\lambda(y)dy\\&+ \int_{E\cap B(z, 1)}H(z, y, 1)\lambda(y)dy.\end{aligned}\end{equation} By heat kernel estimate of Li-Yau \cite{[LY]}, $H(z, y, 1)\geq c(n)>0$ for $y\in B(z, 1)$. Also, with volume comparison,  we find $H(z, y, 1)\leq C(n)$ for $y, z\in B(x, 4R)$.
\begin{equation}\label{eq14}\begin{aligned}\int_{B(x, 2R)\backslash E} |H(z, y, 1)\lambda(y)|dy&\leq C(n)\int_{B(x, 2R)\backslash E}|\lambda(y)|dy\\& \leq C(n)(\int_{B(x, 2R)\backslash E}|\lambda^2(y)|dy)^{\frac{1}{2}} (vol(B(x, 2R)\backslash E))^{\frac{1}{2}}\\&\leq \Phi(\epsilon|n, R)\end{aligned}\end{equation} \begin{equation}\label{eq15}\int_{E\cap B(z, 1)}H(z, y, 1)\lambda(y)dy\geq\frac{1}{2}\int_{E\cap B(z, 1)}H(z, y, 1)dy \geq c(n)>0.\end{equation}

Note $d(y, z)\geq R$ for $y\in B(x, 4R)\backslash B(x, 2R)$. Heat kernel estimate says $H(y, z, 1)\leq C(n)e^{-\frac{R^2}{5}}$. Therefore, by (\ref{eq12}), \begin{equation}\label{eq16}\int_{B(x, 4R)\backslash B(x, 2R)} |H(z, y, 1)\lambda(y)|dy\leq C(n)e^{-\frac{R^2}{5}}R^{2n}<\Phi(\frac{1}{R}).\end{equation}
Putting (\ref{eq14}), (\ref{eq15}), (\ref{eq16}) in (\ref{eq13}), we find \begin{equation}\label{eq17}\lambda(z, 1) = \int H(z, y, 1)\lambda(y)dy\geq c(n)-\Phi(\frac{1}{R}|n)-\Phi(\epsilon|n, R)\end{equation} for $z\in B(x, \frac{R}{10})$. We first let $R$ be large, then $\epsilon$ be very small. Then $\lambda(z, 1)>c(n)$. We conclude the proof of the claim from theorem \ref{thm6}.
\end{proof}

 Recall $u_t$ is defined in (\ref{eq8}). We claim that there exists $\epsilon_0 = \epsilon_0(n) > 0$ so that for large $R$, \begin{equation}\label{eq18}\min\limits_{y\in\partial B(x, \frac{R}{20})}u_1(y)> 4\max\limits_{y\in B(x, \epsilon_0\frac{R}{20})}u_1(y).\end{equation}
This is a simple exercise by using the heat kernel estimate. One can also apply proposition \ref{prop5} to conclude the proof.
From now on, we freeze the value of $R$. That is to say, $R=R(n)>0$ satisfies claim \ref{cl1} and (\ref{eq18}) and $\frac{R}{20}\epsilon_0>100$.

Let $\Omega$ be the connected component of $\{y\in B(x, \frac{R}{20})|u_1(y)<2\max\limits_{y\in B(x, \epsilon_0\frac{R}{20})}u_1(y)\}$ containing $B(x, \epsilon_0\frac{R}{20})$. Then $\Omega$ is relatively compact in $B(x, \frac{R}{20})$ and $\Omega$ is a Stein manifold by claim \ref{cl1}.

Now we apply theorem \ref{thm7} to $\Omega$, with the K\"ahler metric induced from $M$. Take smooth $(0, 1)$ forms $g_s = \overline\partial w' _s$ defined in (\ref{eq4}); the weight function $\psi = u_1$. We find smooth functions $f_s$ in $\Omega$ with $\overline\partial f_s = g_s$ and \begin{equation}\label{eq19}\int_\Omega |f_s|^2e^{-\psi}\omega^n\leq \int_\Omega\frac{|g_s|^2}{c}e^{-\psi}\omega^n\leq \frac{\int_\Omega|\overline\partial w'_s|^2\omega^n}{c(n)}\leq \Phi(\epsilon|n).\end{equation} Here we used the fact that $r=100R = 100R(n)$. By proposition \ref{prop5}, we find $\psi=u_1\leq C(R, n) = C(n)$ in $B(x, R)$. Therefore \begin{equation}\label{eq20}\int_{B(x, 10)}|f_s|^2\omega^n\leq  \int_{\Omega}|f_s|^2\omega^n\leq \Phi(\epsilon|n).\end{equation}

Note $w_s = w'_s - f_s$ is holomorphic, as $\overline\partial w_s = \overline\partial w'_s - g_s = 0$. Since $w'_s$ is harmonic(complex), $f_s$ is also harmonic.
By the mean value inequality \cite{[LS]} and Cheng-Yau's gradient estimate \cite{[CY]}, we find that in $B(x, 5)$, \begin{equation}\label{eq21}|f_s|\leq \Phi(\epsilon|n); |\nabla f_s|\leq \Phi(\epsilon|n).\end{equation}Therefore, (\ref{eq1}) implies \begin{equation}\label{eq22}\int_{B(x, 4)}|(w_s)_i\overline{(w_t)_j}g^{i\overline{j}}-2\delta_{st}|\leq \Phi(\epsilon|n).\end{equation} 
\begin{claim}\label{cl2}
$w_s (s = 1,...., n)$ is a holomorphic chart in $B(x, 1)$.
\end{claim}
\begin{proof}
Recall that $(b_1, ..., b_{2n})$ is an $\Phi(\epsilon|n)$ Gromov Hausdorff approximation to the image in $\mathbb{R}^{2n}$. According to (\ref{eq21}), on $B(x, 1)$, $w = (w_1, ...., w_n)$ is also a $\Phi(\epsilon|n)$ Gromov Hausdorff approximation to $B_{\mathbb{C}^n}(0, 1)$. Thus $w^{-1}(\overline{B_{\mathbb{C}^n}(0, 1)})$ is compact in $B(x, 1+\Phi(\epsilon|n))$. First we prove the degree $d$ of the map $w$ is $1$. By (\ref{eq22}) and that holomorphic maps preserves the orientation, $d\geq 1$.  
We also have \begin{equation}\label{eq23}\begin{aligned}d\cdot vol(B_{\mathbb{C}^n}(0, 1)) &= \frac{1}{(-2\sqrt{-1})^n}\int_{w^{-1}(B(0, 1))}dw_1\wedge d\overline{w_1}\wedge\cdot\cdot\cdot\wedge dw_n\wedge d\overline{w_n}\\&\leq (1+\Phi(\epsilon|n))Vol(B(x, (1+\Phi(\epsilon|n)))+\Phi(\epsilon|n)\end{aligned}\end{equation} by (\ref{eq21}) and (\ref{eq22}). This means that if $\epsilon$ is sufficiently small, $d = 1$. That is to say $(w_1, ... ,w_n)$ is generically one to one in $B(x, 1)$.
Moreover, $(w_1, ...., w_n)$ must be a finite map: the preimage of a point must be a subvariety  which is compact in the Stein manifold $\Omega$, thus finitely many points. According to Remmert's theorem in complex analytic geometry, this is an isomorphism.
\end{proof}
We can make a small perturbation so that $w_s(x) = 0$ for $1\leq s\leq n$. This completes the proof of proposition \ref{prop1}.
\end{proof}

\section{\bf{A gap theorem for complex structure of $\mathbb{C}^n$}}
As the first application of proposition \ref{prop1}, we prove a gap theorem for the complex structure of $\mathbb{C}^n$. The conditions are rather restrictive. However, we shall expand some of the ideas in later sections.
\begin{theorem}
Let $M^n$ be a complete noncompact K\"ahler manifold with nonnegative bisectional curvature and $p\in M$. There exists $\epsilon(n)>0$ so that if $\epsilon<\epsilon(n)$ and \begin{equation}\label{eq24}\frac{vol(B(p, r))}{r^{2n}}\geq \omega_{2n}-\epsilon\end{equation} for all $r>0$, then $M$ is biholomorphic to $\mathbb{C}^n$. Here $\omega_{2n}$ is the volume of the unit ball in $\mathbb{C}^n$. Furthermore, the ring $\mathcal{O}_P(M)$ is finitely generated. In fact, it is generated by $n$ functions which form a coordinate in $\mathbb{C}^n$.
\end{theorem}
\begin{proof}
By scaling if necessary, we may assume $B(p, 1)$ is $C^2$ close to the Euclidean ball $B(0, 1)$ in $\mathbb{R}^{2n}$. In particular, $\partial B(p, 1)$ is $C^2$ close to the sphere $\mathbb{S}^{2n-1}$. Consider the blow down sequence $(M_i, p_i, g_i) = (M, p, \frac{1}{s_i^2}g)$  for $s_i\to\infty$. According to proposition \ref{prop1} and the Cheeger-Colding theory \cite{[CC1]}, if $\epsilon$ is sufficiently small, there exists a holomorphic chart $(w^i_1, ...., w^i_n)$ on $B(p_i, 1)$. Moreover, the map $(w^i_1, ...., w^i_n)$ is a $\Phi(\epsilon|n)$ Gromov-Hausdorff approximation to $B_{\mathbb{C}^n}(0, 1)$.
We may assume \begin{equation}\label{eq25}w^i_s(p_i) = 0\end{equation} for $s = 1,..., n$. We can also regard $w^i_s$ as holomorphic functions on $B(p, s_i)\subset M$. For each $i$, we can find a new basis $v^i_s$ for $span\{w^i_s\}$ so that \begin{equation}\label{eq26}\int_{B(p, 1)}v^i_s\overline{v^i_t} = \delta_{st}.\end{equation} Set \begin{equation}\label{eq27}M^i_s(r) = \max\limits_{x\in B(p, r)}|v^i_s(x)|.\end{equation} 
\begin{claim}\label{cl3}
$\frac{M^i_s(s_i)}{M^i_s(\frac{1}{2}s_i)}\leq 2+\Phi(\epsilon|n)$ for $1\leq s\leq n$.
\end{claim}
\begin{proof}
It suffices to prove this for $s=1$. Let $v^i_1 = \sum\limits_{j=1}^{n} c^i_jw^i_j$. Without loss of generality, assume $|c^i_1| = \max\limits_{1\leq j\leq n}|c^i_j|>0$. Then $\frac{v^i_1}{c^i_1} = w^i_1 + \sum\limits_{j=2}^n\alpha_{ij}w^i_j$ and $|\alpha_{ij}|\leq 1$. Since on $M_i$, $(w^i_1, ..., w^i_n)$ is a $\Phi(\epsilon|n)$ Gromov Hausdorff approximation to $B_{\mathbb{C}^n}(0, 1)$, we find \begin{equation}\label{eq28}\begin{aligned}\frac{M^i_s(s_i)}{M^i_s(\frac{1}{2}s_i)}&=\frac{\max\limits_{x\in B(p, s_i)}|w^i_1(x) + \sum\limits_{j=2}^n\alpha_{ij}w^i_j(x)|}{\max\limits_{x\in B(p, \frac{s_i}{2})}|w^i_1(x) + \sum\limits_{j=2}^n\alpha_{ij}w^i_j(x)|}\\&=\frac{\max\limits_{x\in B(p_i, 1)}|w^i_1(x) + \sum\limits_{j=2}^n\alpha_{ij}w^i_j(x)|}{\max\limits_{x\in B(p_i, \frac{1}{2})}|w^i_1(x) + \sum\limits_{j=2}^n\alpha_{ij}w^i_j(x)|}\\&\leq 2+\Phi(\epsilon|n).\end{aligned}\end{equation} This concludes the proof. \end{proof}

According to the three circle theorem \ref{thm8}, $\frac{M^i_s(2r)}{M^i_s(r)}$ is monotonic increasing for $0<r<\frac{1}{2}s_i$. Then claim \ref{cl3} implies 
\begin{equation}\label{eq29}
\frac{M^i_s(2r)}{M^i_s(r)}\leq 2+\Phi(\epsilon|n)\end{equation}
for $0<r<\frac{1}{2}s_i$. From (\ref{eq26}), we find $M^i_s(\frac{1}{2})\leq C(n)$. (\ref{eq29}) implies \begin{equation}\label{eq30}M^i_s(r)\leq C(n)(r^\alpha+1)\end{equation} for $\alpha = 1+\Phi(\epsilon|n)$.
As $s_i\to\infty$, by taking subsequence, we can assume $v^i_s\to v_s$ uniformly on each compact set of $M$. Set \begin{equation}\label{eq31}M_s(r) = \max\limits_{x\in B(p, r)}|v_s(x)|.\end{equation} Then \begin{equation}\label{eq32}M_s(r)\leq C(n)(r^\alpha+1)\end{equation} for $\alpha = 1+\Phi(\epsilon|n)$ and $r\geq 0$. We may assume $v_s\in\mathcal{O}_{\frac{3}{2}}(M)$. Note that $v_s$ also satisfies \begin{equation}\label{eq33}v_s(p) = 0(1\leq s\leq n); \int_{B(p, 1)}v_s\overline{v_t} = \delta_{st}.\end{equation} Our goal is to prove $(v_1, ..., v_n)$ is a biholomorphism from $M$ to $\mathbb{C}^n$. 
\begin{claim}\label{cl4}
Let $\epsilon$ in (\ref{eq24}) be sufficiently small(depending only on $n$). If we rescale each $v_s$ so that $\max\limits_{B(p, 1)}|v_s| = 1$, then in $B(p, 1)$, $(v_1, ...., v_n)$ is a $\frac{1}{100n}$-Gromov-Haudorff approximation to $B_{\mathbb{C}^n}(0, 1)$.
\end{claim}
\begin{proof}
We argue by contradiction. Assume a positive sequence $\epsilon_i\to 0$ and $(M'_i, q_i)$ is a sequence of $n$-dimensional complete noncompact K\"ahler manifolds with nonnegative bisectional curvature and \begin{equation}\label{eq34}\frac{vol(B(q_i, r))}{r^{2n}}\geq\omega_{2n}-\epsilon_i\end{equation} for all $r> 0$. Assume there exist holomorphic functions $u^i_s (s= 1, ..., n)$ on $M_i'$ so that \begin{equation}\label{eq35}u^i_s(q_i) = 0; u^i_s\in\mathcal{O}_{\frac{3}{2}}(M'_i); \int_{B(q_i, 1)}u^i_s\overline{u^i_t} = c^i_{st}\delta_{st}; \max\limits_{B(q_i, 1)}|u^i_s| = 1.\end{equation}  
Here $c^i_{st}$ are constants. Assume in $B(q_i, 1)$, $(u^i_1, ..., u^i_n)$ is \textbf{not} a $\frac{1}{100n}$-Gromov-Haudorff approximation to $B_{\mathbb{C}^n}(0, 1)$. According to Cheeger-Colding theory \cite{[CC1]} and (\ref{eq34}), $(M'_i, q_i)$ converges to $(\mathbb{R}^{2n}, 0)$ in the pointed Gromov-Hausdorff sense.
By the three circle theorem and (\ref{eq35}), we have uniform bound for $u^i_s$ in $B(q_i, r)$ for any $r>0$. Let $i\to\infty$, there is a subsequence so that $u^i_s\to u_s$ uniformly on each compact set.  Moreover, by remark $9.3$ of \cite{[CCT]}(see also ($21$) in \cite{[L2]}),
there is a natural linear complex structure on $\mathbb{R}^{2n}$. Thus we can identify the limit space with $\mathbb{C}^n$. By lemma $4$ in \cite{[L2]}, the limit of holomorphic functions are still holomorphic. Moreover, $\{u_s\}$ satisfy (\ref{eq35}), according to the three circle theorem. 
Thus $u_s$ are all linear functions which form a standard complex coordinate in $\mathbb{C}^n$. Therefore in $B_{\mathbb{C}^n}(0, 1)$, $(u_1, ..., u_n)$ is an isometry to $B_{\mathbb{C}^n}(0, 1)$. This contradicts the assumption that $(u^i_1, ..., u^i_n)$ is not a $\frac{1}{100n}$-Gromov-Haudorff approximation to $B_{\mathbb{C}^n}(0, 1)$.\end{proof}

Recall that $B(p, 1)$ is $C^2$ close to the Euclidean ball $B(0, 1)$ and $\partial B(p, 1)$ is $C^2$ close to the sphere $\mathbb{S}^{2n-1}$. By the degree theory and claim \ref{cl4}, we find that the degree of the map $(v_1, ...., v_n)$ in $B(p, 1)$ is $1$. This means $dv_1\wedge\cdot\cdot\cdot\wedge dv_n$ is not identically zero.  By (\ref{eq32}) and Cheng-Yau's gradient estimate, $|dv_i|\leq C(n)(r^{\Phi(\epsilon|n)}+1)$. Thus $|dv_1\wedge\cdot\cdot\cdot\wedge dv_n|\leq C(n)(r^{\Phi(\epsilon|n)}+1)$. The canonical line bundle $K_M$ has nonpositive curvature. Note by the remark following corollary \ref{cor2}, three circle theorem also holds for holomorphic sections of nonpositive bundles. Therefore, if the holomorphic $n$-form $dv_1\wedge\cdot\cdot\cdot\wedge dv_n$ vanishes at some point in $M$, then $|dv_1\wedge\cdot\cdot\cdot\wedge dv_n|$ must be of at least linear growth, by corollary \ref{cor2}. Therefore, $dv_1\wedge\cdot\cdot\cdot\wedge dv_n$ is vanishing identically on $M$. This is a contradiction.

Next we prove the map $(v_1, ...., v_n)$: $M\to \mathbb{C}^n$ is proper. Given any $R>1$, we can define a norm $|\cdot|_R$ for the span of $v_1, ..., v_n$ induced by $\int_{B(p, R)}v_s\overline{v_t}$.  There exists a basis $v_1^R, ...., v_n^R$ for the span of $v_1, ..., v_n$ so that \begin{equation}\label{eq36}\int_{B(p, 1)}v_s^R\overline{v_t^R} = \delta_{st}; \int_{B(p, R)}v_s^R\overline{v_t^R}= c(R)_{st}\delta_{st}.\end{equation} Here $c(R)_{st}$ are constants. That is, we diagonalize the two norms $|\cdot|_1$ and $|\cdot|_R$ simultaneously. Obviously we have \begin{equation}\label{eq37}\sum\limits_{s=1}^n|v_s(x)|^2 = \sum\limits_{s=1}^n|v_s^R(x)|^2\end{equation} for any $x\in M$. To prove $(v_1, ..., v_n)$ is proper, it suffices to prove $\sum\limits_{s=1}^n|v_s^R(x)|^2$ is large for $x\in\partial B(p, R)$ and large $R$. Define \begin{equation}\label{eq38}w_s^R(x) = \frac{v_s^R(x)}{c_s^R}\end{equation} where $c_s^R$ are positive constants so that \begin{equation}\label{eq39}\max_{x\in B(p, R)}w_s^R(x) = 1\end{equation} for $s= 1,...., n$. Note $\int_{B(p, 1)}|v_s^R|^2 = 1$ and $v_s^R(p) = 0$. According to corollary \ref{cor2} and (\ref{eq36}), \begin{equation}\label{eq40}c_s^R\geq cR,\end{equation} where $c=c(n) >0$, $R>1$.  We can apply claim \ref{cl4} to $v_s^R$ in $B(p, R)$. Here we have to rescale the radius to $1$. Then we obtain that $(Rw_1^R, ...., Rw_n^R)$ is a $\frac{R}{100n}$-Gromov-Hausdorff approximation from $B(p, R)$ to $B_{\mathbb{C}^n}(0, R)$. In particular, for any $x\in \partial B(p, R)$,  there exists some $s_0$ with $|w_{s_0}^R(x)|\geq \frac{1}{2n}$. Then \begin{equation}\label{eq41}|v_{s_0}^R(x)|= c_{s_0}^R|w_{s_0}^R(x)|\geq \frac{1}{2n}cR;
\sum\limits_{s=1}^n|v_s(x)|^2 = \sum\limits_{s=1}^n|v_s^R(x)|^2\geq c(n)R^2.\end{equation} The properness is proved.

As $dv_1\wedge\cdot\cdot\cdot\wedge dv_n$ is not vanishing at any point on $M$ and $(v_1, ...., v_n)$ is a proper map to $\mathbb{C}^n$, we conclude that $(v_1, ...., v_n)$ is a biholomorphism from $M$ to $\mathbb{C}^n$.

Next we prove $\mathcal{O}_P(M)$ is generated by $(v_1, ...., v_n)$.
We can regard $(v_1,..., v_n)$ as a holomorphic coordinate system on $M$.
If $f\in\mathcal{O}_d(M)$, we can think $f = f(v_1, ...., v_n)$.
It suffices to prove the right hand side is a polynomial. Indeed, $|f(x)|\leq C(1+d(x, p)^d)$.
Note by (\ref{eq41}),  $|f(v_1, ..., v_n)|\leq C((\sum\limits_{s=1}^n|v_s|^2)^{\frac{d}{2}}+1)$.
This proves $f$ is a polynomial of $v_1, ...., v_n$.

\end{proof}
\section{\bf{Proof of theorem \ref{thm3}}}
\begin{proof}
We only consider the case for $n\geq 2$. Otherwise, the result is known. Pick $p\in M$.
Let \begin{equation}\label{eq42}\alpha = \lim\limits_{r\to\infty}\frac{vol(B(p, r))}{r^{2n}}>0.\end{equation}
 Consider the blow down sequence $(M_i, p_i, g_i) = (M, p, \frac{1}{s_i^2}g)$ for $s_i\to\infty$. By Cheeger-Colding theory \cite{[CC1]}, a subsequence converges to a metric cone $(X, p_\infty, d_\infty)$. Define \begin{equation}\label{eq43}r(x) = d_\infty(x, p_\infty), x\in X; r_i(x) = d_{g_i}(x, p_i), x\in M_i.\end{equation} Now pick two regular points $y_0, z_0\in X$ with \begin{equation}\label{eq44}r(y_0) = r(z_0) = 1; d_\infty(y_0, z_0)\geq c(n, \alpha)>0.\end{equation} Note the latter inequality is guaranteed by theorem \ref{thm5}. There exists $\delta_0>0$ satisfying \begin{equation}\label{eq45}B(y_0, 2\delta_0)\cap B(z_0, 2\delta_0) = \Phi; \delta_0<\frac{1}{10};\end{equation} \begin{equation}\label{eq46}d_{GH}(B(y_0, \frac{1}{\epsilon}\delta_0), B_{\mathbb{R}^{2n}}(0, \frac{1}{\epsilon}\delta_0))\leq \frac{1}{2}\epsilon\delta_0; d_{GH}(B(z_0, \frac{1}{\epsilon}\delta_0), B_{\mathbb{R}^{2n}}(0, \frac{1}{\epsilon}\delta_0))\leq \frac{1}{2}\epsilon\delta_0.\end{equation} Here $\epsilon = \frac{1}{2}\epsilon(n)$, which is given by proposition \ref{prop1}.
Therefore, if $i$ is sufficiently large, we can find points $y_i, z_i\in M_i$ with $r_i(y_i)=r_i(z_i) = 1$ and 
\begin{equation}\label{eq47}d_{GH}(B(y_i, \frac{1}{\epsilon}\delta_0), B_{\mathbb{R}^{2n}}(0, \frac{1}{\epsilon}\delta_0))\leq \epsilon\delta_0; d_{GH}(B(z_i, \frac{1}{\epsilon}\delta_0), B_{\mathbb{R}^{2n}}(0, \frac{1}{\epsilon}\delta_0))\leq \epsilon\delta_0.\end{equation}
Let $w^i_s$ and $v^i_s$ be the local holomorphic charts around $y_i$ and $z_i$ constructed in proposition \ref{prop1}. Note that they have uniform size (independent of $i$). By changing the value of $\delta_0$, we may assume $w^i_s, v^i_s$ are holomorphic charts in $B(y_i, \delta_0)$ and $B(z_i, \delta_0)$. Moreover, \begin{equation}\label{eq48}|dw^i_s|, |dv^i_s|\leq C(n); w^i_s(y_i)= 0, v^i_s(z_i)=0;\end{equation} \begin{equation}\label{eq49}|\sum\limits_{s= 1}^{n}|w^i_s(y)|^2 - d_{g_i}(y, y_i)^2|\leq \Phi(\epsilon|n)\delta_0^2;\end{equation}   \begin{equation}\label{eq50}|\sum\limits_{s= 1}^{n}|v^i_s(z)|^2 - d_{g_i}(z, z_i)^2|\leq \Phi(\epsilon|n)\delta_0^2\end{equation} for $y\in B(y_i, \delta_0), z\in B(z_i, \delta_0)$. We need to construct a weight function on $B(p_i, R)$ for some large $R$ to be determined later. The construction is similar to proposition \ref{prop1}. Set \begin{equation}\label{eq51}A_i = B(p_i, 5R)\backslash B(p_i, \frac{1}{5R}).\end{equation}   By Cheeger-Colding theory \cite{[CC1]}((4.43) and (4.82)), there exists a smooth function $\rho_i$ on $M_i$ so that 
\begin{equation}\label{eq52}
\int_{A_i}|\nabla\rho_i-\nabla \frac{1}{2}r_i^2|^2 + |\nabla^2\rho_i-g_i|^2<\Phi(\frac{1}{i}|R);
\end{equation}
\begin{equation}\label{eq53}
|\rho_i-\frac{r_i^2}{2}|<\Phi(\frac{1}{i}|R)
 \end{equation} in $A_i$. According to ($4.20$)-($4.23$) in \cite{[CC1]}, \begin{equation}\label{eq54}\rho_i =  \frac{1}{2}(\mathcal{G}_i)^{\frac{2}{2-2n}}; \Delta\mathcal{G}_i(x) = 0, x\in B(p_i, 10R)\backslash B(p_i, \frac{1}{10R});\end{equation}  \begin{equation}\label{eq55}\mathcal{G}_i = r_i^{2-2n}\end{equation} on $\partial (B(p_i, 10R)\backslash B(p_i, \frac{1}{10R})).$ Now \begin{equation}\label{eq56}|\nabla\rho_i(y)| = C(n)|\mathcal{G}_i|^{\frac{n}{1-n}}|\nabla\mathcal{G}_i(y)|.\end{equation} By (\ref{eq53})-(\ref{eq55}) and Cheng-Yau's gradient estimate, \begin{equation}\label{eq57}|\nabla\rho_i(y)|\leq C(n)r_i(y)\end{equation} for $y\in A_i$ and sufficiently large $i$. Now consider a smooth function $\overline\varphi$: $\mathbb{R}^+\to \mathbb{R}^+$ given by $\overline\varphi(t) = t$ for $t\geq 2$; $\overline\varphi(t) = 0$ for $0\leq t\leq 1$; $|\overline\varphi|, |\overline\varphi'|, |\overline\varphi''|\leq C(n)$. Let \begin{equation}\label{eq58}u_i(x) = \frac{1}{R^2}\overline\varphi(R^2\rho_i(x)).\end{equation}  We set $u_i(x) = 0$ for $x\in B(p_i, \frac{1}{5R})$. Then $u_i$ is smooth in $B(p_i, 4R)$.
 \begin{claim}\label{cl5}
 For sufficiently large $i$, $\int_{B(p_i, 4R)}|\nabla u_i-\nabla \frac{1}{2}r_i^2|^2 + |\nabla^2u_i-g_i|^2<\Phi(\frac{1}{R}); |u_i-\frac{r_i^2}{2}|<\Phi(\frac{1}{R})$ and $|\nabla u_i|\leq C(n)r_i$ in $B(p_i, 4R)$.
  \end{claim}
  \begin{proof}
  We have \begin{equation}\label{eq59}\nabla u_i(x) = \overline\varphi'(R^2\rho_i(x))\nabla\rho_i(x);\end{equation}\begin{equation}\label{eq60}\nabla^2 u_i(x) = R^2\overline\varphi''(R^2\rho_i(x))\nabla\rho_i\otimes\nabla\rho_i + \overline\varphi'(R^2\rho_i(x))\nabla^2\rho_i.\end{equation}
The proof follows from a routine calculation, by (\ref{eq53}), (\ref{eq54}), (\ref{eq57}).
  \end{proof}
  
Similar as in proposition \ref{prop1}, consider a smooth function $\varphi$: $\mathbb{R}^+\to \mathbb{R}^+$ with $\varphi(t) = t$ for $0\leq t\leq 1$; $\varphi(t) = 0$ for $t\geq 2$; $|\varphi|, |\varphi'|, |\varphi''|\leq C(n)$. Set \begin{equation}\label{eq61} v_i(z)=3R^2\varphi(\frac{u_i(z)}{3R^2}), v_{i, t}(z) = \int_M H_i(z, y, t)v_i(y)dy.\end{equation} Here $H_i(x, y, t)$ is the heat kernel on $M_i$. Then $v_i$ is supported in $B(p_i, 4R)$.
By similar arguments as in claim \ref{cl1}, we arrive at the following:
 \begin{prop}\label{prop2}
 $v_{i, 1}(z)$ satisfies that $(v_1)_{\alpha\overline\beta}(z)\geq c(n, \alpha)g_{\alpha\overline\beta}>0$ for $z\in B(p_i, \frac{R}{10})$. Here $\alpha>0$ is given by (\ref{eq42}).\end{prop}
 
 Now define \begin{equation}\label{eq62}q_i(x) = 4n(\log(\sum\limits_{s = 1}^{n}|w^i_s|^2)\lambda(4\frac{\sum\limits_{s = 1}^{n}|w^i_s|^2}{\delta_0^2}) + \log(\sum\limits_{s = 1}^{n}|v^i_s|^2)\lambda(4\frac{\sum\limits_{s = 1}^{n}|v^i_s|^2}{\delta_0^2})).\end{equation} Here $\lambda$ is a standard cut-off function $\mathbb{R}^+\to \mathbb{R}^+$ with $\lambda(t) = 1$ for $0\leq t\leq 1$; $\lambda(t) = 0$ for $t\geq 2$.  Note by (\ref{eq49}) and (\ref{eq50}), $q_i(x)$ has compact support in $B(y_i, \delta_0)\cup B(z_i, \delta_0)\subset B(p_i, 2)$.
 \begin{lemma}\label{lm1}
 $\sqrt{-1}\partial\overline\partial q_i \geq -C(n, \delta_0)\omega_i$. Moreover, $e^{-q_i(x)}$ is not locally integrable at $y_i$ and $z_i$.
 \end{lemma}
 \begin{proof}
\begin{equation}\label{eq63}|\sqrt{-1}\partial\overline\partial |w^i_s|^2| = |\partial w^i_s\wedge\overline{\partial w^i_s}|\leq |dw^i_s|^2\leq C(n)\end{equation} in $B(y_i, \delta_0)$. When $\lambda'(4\frac{\sum\limits_{s = 1}^{n}|w^i_s|^2}{\delta_0^2})\neq 0$, \begin{equation}\label{eq64}\delta_0^2\geq\sum\limits_{s = 1}^{n}|w^i_s|^2\geq \frac{1}{4}\delta_0^2.\end{equation}
Also note \begin{equation}\label{eq65}\sqrt{-1}\partial\overline\partial\log(\sum\limits_{s = 1}^{n}|w^i_s|^2)\geq 0\end{equation} in the current sense. Then the proof of the first part follows from routine calculation.

 For the second part, when $x\in B(y_i, \frac{\delta_0}{10})$, $e^{-q_i(x)} = \frac{1}{(\sum\limits_{s = 1}^{n}|w^i_s|^2)^{4n}}$. As $w^i_s(y_i) = 0$ for all $s$, a simple calculation shows $e^{-q_i(x)}$ is not locally integrable at $y_i$. The same argument works for $z_i$.
  \end{proof}
  Putting proposition \ref{prop2} and lemma \ref{lm1} together, we find $C(n, \alpha, \delta_0)>0$ so that \begin{equation}\label{eq66}\sqrt{-1}\partial\overline\partial(q_i(x) + C(n, \alpha, \delta_0)v_{i, 1}(x))\geq \omega_i\end{equation} in $B(p_i, \frac{R}{15})$.  Set \begin{equation}\label{eq67}\psi_i(x) = q_i(x) + C(n, \alpha, \delta_0)v_{i, 1}(x).\end{equation}
    
  By the same argument as in proposition \ref{prop1}, we find $\epsilon_0 = \epsilon_0(\alpha, n) > 0$ so that for sufficiently large $R$, \begin{equation}\label{eq68}\min\limits_{y\in\partial B(p_i, \frac{R}{20})}v_{i, 1}(y)> 4\max\limits_{y\in B(p_i, \epsilon_0\frac{R}{20})}v_{i, 1}(y).\end{equation} Of course, we can assume \begin{equation}\label{eq69}\frac{\epsilon_0R}{20}>4.\end{equation}
  From now on, we freeze the value of $R$. That is, \begin{equation}\label{sq70}R=R(n, \alpha)>0\end{equation} satisfies the all the conditions above. 
Let $\Omega_i$ be the connected component of $\{y\in B(p_i, \frac{R}{20})|v_{i, 1}(y)<2\max\limits_{y\in B(p_i, \epsilon_0\frac{R}{20})}v_{i, 1}(y)\}$ containing $B(p_i, \epsilon_0\frac{R}{20})$. Then $\Omega_i$ is relatively compact in $B(p_i, \frac{R}{20})$ and $\Omega_i$ is a Stein manifold, by proposition \ref{prop2}. Also $B(p_i, 3)\subset\Omega_i$.

 Now consider a function $f_i(x) = 1$ for $x\in B(y_i, \frac{\delta_0}{4})$; $f_i$ has compact support in $B(y_i, \delta_0)\subset B(p_i, 2)$; $|\nabla f_i|\leq C(n, \alpha, \delta_0)$. We solve the equation $\overline\partial h_i = \overline\partial f_i$ in $\Omega_i$ with \begin{equation}\label{eq71}\int_{\Omega_i}|h_i|^2e^{-\psi_i}\leq \int_{\Omega_i}|\overline\partial f_i|^2e^{-\psi_i}\leq C(n, \alpha, \delta_0).\end{equation} By lemma \ref{lm1}, $h_i(y_i) = h_i(z_i) = 0$. Therefore, the holomorphic function $\mu_i = f_i-h_i$ is not constant in $\Omega_i$. It is easy to see that $\psi_i(x)\leq C(n,\alpha, \delta_0)$ in $B(p_i, 3)$. Then \begin{equation}\label{eq72}\frac{1}{C(n, \alpha, \delta_0)}\int_{B(p_i, 3)}|h_i|^2 \leq \int_{\Omega_i}|h_i|^2e^{-\psi_i}\leq C(n, \alpha, \delta_0).\end{equation} Thus \begin{equation}\label{eq73}\int_{B(p_i, 3)}|\mu_i|^2 \leq 2\int_{B(p_i, 3)}(|h_i|^2 + |f_i|^2)\leq C(n, \alpha, \delta_0).\end{equation}
 Mean value inequality implies that \begin{equation}\label{eq74}|\mu_i(x)|\leq C(n, \alpha, \delta_0)\end{equation} for $x\in B(p_i, 2)$. Therefore, the holomorphic function \begin{equation}\label{eq75}\nu^*_i(x) = \mu_i(x)-\mu_i(p_i)\end{equation} is uniformly bounded in $B(p_i, 2)$. Set \begin{equation}\label{eq76}M'_i(r) = \max\limits_{x\in B(p_i, r)}|\nu_i(x)|.\end{equation} Then \begin{equation}\label{eq77}M'_i(2)\leq C(n, \alpha, \delta_0).\end{equation} On the other hand, as $\mu_i(y_i) = f_i(y_i) - h_i(y_i) = 1$ and $\mu_i(z_i)= f_i(z_i)-h_i(z_i) = 0$, we find \begin{equation}\label{eq78}M'_i(1)\geq \frac{1}{2}.\end{equation} Therefore \begin{equation}\label{eq79}\frac{M'_i(2)}{M'_i(1)}\leq C(n, \alpha, \delta_0).\end{equation} Now we are ready to apply the three circle theorem. More precisely,
 we consider the rescale functions $\overline\nu^*_i = \beta_i\nu^*_i$ in $B(p, 2s_i)\subset M$. Here $\beta_i$ are constants so that \begin{equation}\label{eq80}\int_{B(p, 2)}|\overline\nu^*_i|^2 = 1.\end{equation} This implies \begin{equation}\label{eq81}|\overline\nu^*_i|\leq C(n, \alpha)\end{equation} in $B(p, 1)$. Set $M_i(r) = \max\limits_{x\in B(p, r)}|\overline\nu^*_i|$. The three circle theorem says $\frac{M_i(2r)}{M_i(r)}$ is monotonic increasing for $0< r\leq s_i$. By (\ref{eq79}) and similar arguments as in (\ref{eq32}),  we obtain that \begin{equation}\label{eq82}M_i(r)\leq C(n, \alpha, \delta_0)(r^{C(n, \alpha, \delta_0)}+1)\end{equation} for all $i$ and $s_i\geq r$. Let $i\to\infty$, a subsequence of $\overline\nu^*_i$ converges uniformly on each compact set to a holomorphic function $v$ of polynomial growth. $v$ cannot be constant, as $v$ satisfies $v(p) = 0$ and $\int_{B(p, 2)}|v|^2 = 1$. Moreover, the degree at infinity is bounded by $C(n, \alpha, \delta_0)$.
 \end{proof}
\begin{remark}
By Gromov compactness theorem, we can find $\delta_0=\delta_0(n, \alpha), y_0, z_0$ satisfying (\ref{eq44}), (\ref{eq45}) and (\ref{eq46}). Therefore, the degree of the holomorphic function at infinity is bounded by $C(n, \alpha)$. The dependence on $\alpha$ is obvious necessary if we look at the complex one dimensional case.
\end{remark}

\begin{cor}
Let $M^n$ be a complete K\"ahler manifold with nonnegative bisectional curvature and maximal volume growth. Then the transcendental dimension of polynomial growth holomorphic functions is $n$. Moreover, $\mathcal{O}_P(M)$ separates points and tangents on $M$.
\end{cor}
\begin{proof}
From theorem \ref{thm3}, there exists a nonconstant holomorphic function $f$ of polynomial growth. First we assume the universal cover of $M$ does not split as products. Then by theorem $3.1$ in \cite{[NT1]},  if we run the heat flow for $\log(|f|^2+1)$, the function becomes strictly plurisubharmonic of logarithmic growth. Then we can apply Hormander's $L^2$ estimate(for example, theorem $5.2$ in \cite{[N1]}) to conclude that $\mathcal{O}_P(M)$ separates points and tangents on $M$. Together with the multiplicity estimate theorem \ref{thm9}, we proved that the transcendental dimension of holomorphic functions of polynomial growth over $\mathbb{C}$ is $n$. 
If the universal covering splits, we work on the universal covering space. Each factor must be of maximal volume growth. Then we can find nonconstant holomorphic functions of polynomial growth. Then we run the heat flow for each factor to obtain strictly plurisubharmonic functions of logarithmic growth. Then we add these function together, which is still strictly plurisubharmonic. Finally, to put these functions back to $M$, just observe that $\pi_1(M)$ is finite, then we can symmetrize the function. Then it projects to $M$, still with logarithmic growth.
Then the argument is the same for the nonsplitting case.

\end{proof}

\begin{remark}
In this case, one can actually prove $M^n$ is biholomorphic to a quasi-affine variety. This follows from Mok's deep work in \cite{[Mo]}. However, with the aid of the theorem below, we shall give a direct proof that $M^n$ is biholomorphic to an affine algebraic variety. 
\end{remark}

\section{\bf{A properness theorem}}
\begin{theorem}\label{thm4}
Let $M^n$ be a complete noncompact K\"ahler manifold with nonnegative bisectional curvature and maximal volume growth. Then there exist finitely many holomorphic functions of polynomial growth $f_1, .... , f_N$ so that $\sum\limits_{i=1}^N|f_i(x)|^2 \geq cd(x, p)^2$. Here $p$ is a fixed point on $M$ and $c>0$ is a constant independent of $x$.
\end{theorem}
\begin{proof}
Put \begin{equation}\label{eq83}v = \lim\limits_{r\to\infty}\frac{vol(B(p, r))}{r^{2n}}>0.\end{equation} 
\begin{prop}\label{prop3}
Let $(Y^n, q)$ be a complete K\"ahler manifold with nonnegative bisectional curvature. Assume $\frac{vol(B(q,  r))}{r^{2n}}\geq v>0$ for all $r>0$. There exists $\epsilon(n, v)>0$ so that if $\epsilon<\epsilon(n, v)$ and if there exists a metric cone $(X, o)$($o$ is the vertex) with \begin{equation}\label{eq84}d_{GH}(B(q, \frac{1}{\epsilon}R), B_{X}(o, \frac{1}{\epsilon}R))\leq \epsilon R,\end{equation} then there exist $N = N(v, n)\in\mathbb{N}, 1 >\delta_1> 5\delta_2>\delta(v, n)>0$ and holomorphic functions $g^1, ..., g^N$ in $B(q, \delta_1R)$ with $g^j(q) = 0$ and \begin{equation}\label{eq85}\min\limits_{x\in\partial B(q, \frac{\delta_1}{3}R)}\sum\limits_{j=1}^N|g^j(x)|^2> 2\max\limits_{x\in B(q, \delta_2R)}\sum\limits_{j=1}^N|g^j(x)|^2.\end{equation} Furthermore, for all $j$, \begin{equation}\label{eq86}\frac{\max\limits_{x\in B(q, \frac{1}{2}\delta_1R)}|g^j(x)|^2}{\max\limits_{x\in B(q, \frac{1}{3}\delta_1R)}|g^j(x)|^2}\leq C(n, v).\end{equation} \end{prop}
\begin{proof}
It is clear the proposition is independent of the value of $R$. Then, by scaling, we may assume $R$ is sufficiently large, to be determined.
Assume \begin{equation}\label{eq87}X= \mathbb{R}^k\times Z.\end{equation} We will do induction on $k$. For the case $k = 2n$, the proposition reduces to proposition \ref{prop1}. Assume the proposition holds for $k = 2s$ and fails for $k = 2s-2$. Then there exist complete K\"ahler manifolds $(Y_i^n, q_i)(i\in\mathbb{N})$ with nonnegative bisectional curvature and $\frac{vol(B(q_i, r))}{r^{2n}}\geq v>0$ for all $r>0$; metric cones $(X_i, o_i)$ with \begin{equation}\label{eq88}(X_i, o_i) = (\mathbb{R}^{2s-2}, 0)\times (Z_i, z_i^*); d_{GH}(B(q_i, iR), B_{X_i}(o_i, iR))\leq \Phi(\frac{1}{i}|R).\end{equation} But the proposition fails to hold uniformly for any subsequence of $Y_i$.  By passing to a subsequence, we may assume $(X_i, o_i)$ converges in the pointed Gromov-Hausdorff sense to a metric cone $(X_0, o_0)$. Of course, there exists a sequence $s_i\to\infty$ with \begin{equation}\label{eq89}(X_0, o_0) = (\mathbb{R}^{2s-2}, 0)\times (Z_0, z_0^*); d_{GH}(B(q_i, s_iR), B_{X_0}(o_0, s_iR))<\Phi(\frac{1}{i}|R).\end{equation} $Z_0$ does not split off a factor $\mathbb{R}^2$, by induction hypothesis.  
Similar to the construction in (\ref{eq61}), we have a function $v_{i, 1}$ so that in $B(q_i, 10R)$, \begin{equation}\label{eq90}\sqrt{-1}\partial\overline\partial v_{i, 1}\geq c(n, v)\omega_i>0;\end{equation}
\begin{equation}\label{eq91}\min\limits_{y\in B(q_i, \frac{1}{2}R)\backslash B(q_i, \frac{1}{4}R)}v_{i, 1}(y)> 4\max\limits_{y\in B(q_i, \delta_3R)}v_{i, 1}(y);\end{equation} \begin{equation}\label{eq92}\min\limits_{y\in B(q_i, \frac{\delta_3}{2}R)\backslash B(q_i, \frac{\delta_3}{4}R)}v_{i, 1}(y)> 4\max\limits_{y\in B(q_i, \delta_4R)}v_{i, 1}(y).\end{equation} Here $\delta_s = \delta_s(n, v)>0(s = 3, 4)$. By proposition \ref{prop5}, we may also assume \begin{equation}\label{eq93}4\max\limits_{y\in B(q_i, \delta_4R)}v_{i, 1}(y)>\frac{1}{2}; \delta_3R>100.\end{equation} Now we freeze the value of $R$. That is to say, \begin{equation}\label{eq94}R=R(n, v)>0.\end{equation} Then \begin{equation}\label{eq95}|v_{i, 1}(y)|\leq C(R, n, v) = C(n, v), y\in B(q_i, R).\end{equation}Let $\Omega_i$ be the connected component of $\{z|v_{i, 1}(z)<\max\limits_{B(q_i, \delta_3R)}v_{i, 1}\}$ containing $B(q_i, \delta_3R)$. As before, we see $\Omega_i$ is Stein.

According to (\ref{eq89}) and ($2.4$)-($2.11$) in \cite{[CCT]}, there exist harmonic functions $b^i_l(1\leq l\leq 2s-2)$ in $B(q_i, 2R)$ with \begin{equation}\label{eq96}\int_{B(q_i, R)}\sum\limits_{1\leq l_1, l_2\leq 2s-2}|\langle\nabla b^i_{l_1}, \nabla b^i_{l_2}\rangle - \delta_{l_1l_2}|^2 +\sum\limits_{l}|\nabla^2 b^i_l|^2<\Phi(\frac{1}{i}|n);\end{equation}\begin{equation}\label{eq97}b^i_l(q_i) = 0; |\nabla b^i_l|\leq C(n) \end{equation} in $B(q_i, R)$. Moreover, in $B(q_i, R)$, $(b^i_1, ..., b^i_{2s-2})$ approximates the $(y_1, ... ,y_{2s-2})$ with error $\Phi(\frac{1}{i}|n)$. Here $(y_1, ..., y_{2s-2})$ is the Euclidean coordinate in $(X_0, o_0) = (\mathbb{R}^{2s-2}, 0)\times (Z_0, z_0^*)$.
By similar arguments as before, we may assume that \begin{equation}\label{eq98}\int_{B(q_i, R)} |J\nabla b^i_{2m-1} - \nabla b^i_{2m}|^2 \leq \Phi(\frac{1}{i}|n)\end{equation}
for $1\leq m\leq s-1$. Set $\tilde{w}^i_m = b^i_{2m-1} +\sqrt{-1}b^i_{2m}$. Then \begin{equation}\label{eq99}\int_{B(q_i, R)}|\overline\partial\tilde{w}^i_m|^2\leq \Phi(\frac{1}{i}|n). \end{equation}  Then by solving $\overline\partial$ problem as before, we find holomorphic functions $w^i_m(1\leq m\leq s-1)$ with \begin{equation}\label{eq100}w^i_m(q_i) = 0; |w^i_m - \tilde{w}^i_m|\leq\Phi(\frac{1}{i}|n)  \end{equation} in $B(q_i, \frac{R}{2})$. Recall $\delta_3$ appeared in (\ref{eq92}).
For sufficiently large $i$, define \begin{equation}\label{eq101}E_i = \{x|x\in\partial B(q_i, \frac{\delta_3R}{3}), \sum\limits_{m=1}^{s-1}|w^i_m|^2 \leq \frac{(\delta_3R)^2}{27}\}; \end{equation} 
  \begin{equation}\label{eq102}E = \{x|x\in\partial B_{\mathbb{R}^{2s-2}\times Z_0}((0, z_0^*), \frac{\delta_3R}{3}), \sum\limits_{k=1}^{2s-2}|y_k|^2 \leq \frac{(\delta_3R)^2}{18}\}.\end{equation} Then the limit of $E_i$ is contained in $E$ under the Gromov-Hausdorff approximation. Observe from the definition and (\ref{eq93}), if $x\in \partial B(q_i, \frac{\delta_3R}{3})\backslash E_i$, \begin{equation}\label{eq103}\sum\limits_{m=1}^{s-1}|w^i_m(x)|^2 > \frac{(\delta_3R)^2}{27}>1.\end{equation}

For $x\in E$, let $C_x$ be a tangent cone. Then $C_x$ must split off a factor $\mathbb{R}^{2s-1}$. Since $C_x$ is the Gromov-Hausdorff limit of K\"ahler manifolds with noncollapsed volume and nonnegative Ricci curvature, $C_x$ splits off a factor $\mathbb{R}^{2s}$, by \cite{[CCT]}. Thus there exists \begin{equation}\label{eq104}
\frac{\delta_3R}{20}>r_x>0\end{equation} with \begin{equation}\label{eq105}d_{GH}(B(x, \frac{1}{\epsilon}r_x), B_{W} (w, \frac{1}{\epsilon}r_x))<\epsilon r_x; (W, w) = (\mathbb{R}^{2s}, 0)\times (H, h^*).\end{equation} Here $\epsilon<\epsilon(n, v)$ satisfies proposition \ref{prop3} for the case $X$ splits off $\mathbb{R}^{2s}$; $(H, h^*)$ is a metric cone with vertex at $h^*$.  By compactness, there is a uniform positive lower bound of $r_x$, say \begin{equation}\label{eq106}r_x\geq R_0(X_0)>0\end{equation} By Gromov compactness, we actually have \begin{equation}\label{eq107}r_x> R_0=R_0(n, v)>0. \end{equation}

Then for sufficiently large $i$ and any point $x_i\in E_i$,  \begin{equation}\label{eq108}d_{GH}(B(x_i, \frac{1}{\epsilon}r_{x_i}), B_{W_i} (w_i, \frac{1}{\epsilon}r_{x_i}))<\epsilon r_{x_i}.\end{equation} \begin{equation}\label{eq109}\frac{\delta_3R}{20}>r_{x_i}\geq R_0>0;  (W_i, w_i)= (\mathbb{R}^{2s}, 0)\times (H_i, h_i^*).\end{equation}  Here $\epsilon$ is the same as in (\ref{eq105}); $(H_i, h_i^*)$ is a metric cone with vertex at $h_i^*$. We apply the induction to $B(x_i, \frac{1}{\epsilon}r_{x_i})$.  By induction hypothesis, there exist \begin{equation}\label{eq110}1>\delta_1>5\delta_2>\delta(n, v); N=N(v, n)\in\mathbb{N}\end{equation} and holomorphic functions $g^j_i(1\leq j\leq N)$ in $B(x_i, \delta_1r_{x_i})$ with 
\begin{equation}\label{eq111}
g_i^j(x_i) = 0; \min_{x\in\partial B(x_i, \frac{\delta_1r_{x_i}}{3})}\sum\limits_{i=1}^N|g^j_i(x)|^2> 2\max_{x\in B(x_i, \delta_2r_{x_i})}\sum\limits_{i=1}^N|g^j_i(x)|^2;\end{equation}
\begin{equation}\label{eq112}
\frac{\max\limits_{x\in B(x_i, \frac{1}{2}\delta_1r_{x_i})}|g_i^j(x)|^2}{\max\limits_{x\in B(x_i, \frac{1}{3}\delta_1r_{x_i})}|g_i^j(x)|^2}\leq C(n, v).
\end{equation} By normalization, we can also assume \begin{equation}\label{eq113}\max\limits_{j}\max\limits_{y\in B(x_i, \delta_2r_{x_i})}|g^j_i(y)| = 2.\end{equation} 

Note by three circle theorem, \begin{equation}\label{eq114}\max\limits_{y\in B(x_i, \frac{r_{x_i}\delta_1}{2})}|g^j_i(y)| \leq C(n, v).\end{equation} Set \begin{equation}\label{eq115}F_i(x) = \sum\limits_{j=1}^N|g^j_i|^2.\end{equation} Let $\lambda$ be a standard cut-off function: $\mathbb{R}^+\to\mathbb{R}^+$ given by $\lambda(t) = 1$ for $0\leq t\leq 1$; $\lambda(t) = 0$ for $t\geq 2$; $|\lambda'|, |\lambda''|\leq C(n)$.
Consider \begin{equation}\label{eq116}h_i(x) = 4n\log F_i(x)\lambda(F_i(x)).\end{equation} By (\ref{eq111}) and (\ref{eq113}), $h_i(x)$ is supported in $B(x_i, \frac{\delta_1r_{x_i}}{3})$. Similar to lemma \ref{lm1}, It is easy to check that \begin{equation}\label{eq117}\sqrt{-1}\partial\overline\partial h_i(x)\geq -C(n, v)\omega_i.\end{equation} Therefore, there exists $\xi = \xi(n, v)>0$ with \begin{equation}\label{eq118}\sqrt{-1}\partial\overline\partial(\xi v_{i, 1}+h_i)\geq \omega_i\end{equation} in $\Omega_i$. We will assume such $\xi$ is large, which will be determined later.  Set \begin{equation}\label{eq119}\phi(x) = \xi v_{i, 1}(x)+h_i(x).\end{equation}  Now consider a function \begin{equation}\label{eq120}\mu_i(x) = \varphi(\frac{d(x, x_i)}{\delta_1r_{x_i}}).\end{equation} Here $\varphi(t) = 1$ for $t\leq \frac{1}{3}$; $\varphi(t) = 0$ for $t\geq 1$ and $|\varphi'|\leq C(n)$. Then it is clear $\mu_i$ is supported in $B(x_i, \delta_1r_{x_i})$. Also, by (\ref{eq109}), \begin{equation}\label{eq121}|\nabla\mu_i|\leq C(n, v).\end{equation}
We solve the $\overline\partial$ problem $\overline\partial s_i = \overline\partial\mu_i$ in $\Omega_i$ satisfying \begin{equation}\label{eq122}\begin{aligned}\int_{\Omega_i}e^{-\phi}|s_i|^2&\leq \int_{\Omega_i}e^{-\phi}|\overline\partial\mu_i|^2 \\&= \int_{B(x_i, \delta_1r_{x_i})\backslash B(x_i, \frac{\delta_1r_{x_i}}{3})}e^{-\phi}|\overline\partial\mu_i|^2\\&\leq \exp(-\xi\min\limits_{y\in B(q_i, \frac{\delta_3R}{2})\backslash B(q_i, \frac{\delta_3R}{4})}v_{i, 1}(y))C(n, v).\end{aligned}\end{equation} 

Here we used that $h_i$ is supported in $B(x_i, \frac{\delta_1r_{x_i}}{3})$. We also used that \begin{equation}\label{eq123}B(x_i, \delta_1r_{x_i})\subset B(q_i, \frac{\delta_3R}{2})\backslash B(q_i, \frac{\delta_3R}{4}),\end{equation} by (\ref{eq109}).
Observe by (\ref{eq92}), $\mu_i$ vanishes in $B(q_i, \frac{1}{2}\delta_4R)$.
Hence $s_i$ is holomorphic in $B(q_i, \frac{1}{2}\delta_4R)$. Mean value inequality implies for $x\in B(q_i, \frac{\delta_4R}{5})$, \begin{equation}\label{eq124}\begin{aligned}|s_i(x)|&\leq \frac{\int_{B(q_i, \delta_4R)}|s_i|^2}{c(n, v)(\delta_4R)^{2n}}\\&\leq \exp(\xi\max\limits_{y\in B(q_i, \delta_4R)}v_{i, 1}(y))\frac{\int_{\Omega_i}e^{-\phi}|s_i|^2}{c(n, v)(\delta_4R)^{2n}}\\&\leq \exp(-\xi(\min\limits_{y\in B(q_i, \frac{\delta_3R}{2})\backslash B(q_i, \frac{\delta_3R}{4})}v_{i, 1}(y)-\max\limits_{y\in B(q_i, \delta_4R)}v_{i, 1}(y)))\frac{1}{c(n, v)(\delta_4R)^{2n}}\\&\leq \frac{e^{-\frac{\xi}{4}}}{c(n, v)(\delta_4R)^{2n}}.\end{aligned}\end{equation}
Here we used (\ref{eq92}) and (\ref{eq93}).
If $\xi$ is large (depending only on $n, v$), then we can make \begin{equation}\label{eq125}|s_i(x)|\leq \frac{1}{10}\end{equation} for $x\in B(q_i, \frac{\delta_4R}{5})$. Now we freeze the value of $\xi = \xi(n, v)$. Note that the local integrability of $s_i$ forces $s_i(x_i) = 0$. Set \begin{equation}\label{eq126}w_i^1(x) = \mu_i(x)-s_i(x).\end{equation} Then \begin{equation}\label{eq127}w_i^1(x_i) = 1; |w_i^1(x)|\leq \frac{1}{10}\end{equation} in $B(q_i, \frac{\delta_4R}{5})$. Set \begin{equation}\label{eq128}f_i^1(x) = w_i^1(x) - w_i^1(q_i).\end{equation} Then \begin{equation}\label{eq129}f_i^1(q_i) = 0; |f_i^1(x_i)|\geq \frac{9}{10}.\end{equation} By (\ref{eq122}), we find \begin{equation}\label{eq130}|f_i^1(x)|\leq C(n, v), |\nabla f_i^1(x)|\leq C(n, v), x\in B(q_i, \frac{2\delta_3R}{3}).\end{equation} 

Therefore, there exists $\delta_5(n, v)>0$ so that \begin{equation}|\label{eq131}f_i^1(x)|\geq \frac{1}{2}\end{equation}  in $B(x_i, \delta_5R)$.
We can take $x^j\in E$ with $j = 1, 2, ... , K$, $K=K(v, n)$. Also \begin{equation}\label{eq132}\cup_j B(x^j, \frac{\delta_5R}{3})\supset E; \frac{\delta_3R}{20}>r_{x^j}\geq R_0(n, v)>0;\end{equation}   \begin{equation}\label{eq133}d_{GH}(B(x^j, \frac{1}{\epsilon}r_{x^j}), B_{W^j} (w^j, \frac{1}{\epsilon}r_{x^j}))<\epsilon r_{x^j}; (W^j, w^j) = (\mathbb{R}^{2s}, 0)\times (H^j, (h^j)^*).\end{equation} Here $(H^j, (h^j)^*)$ is a metric cone with vertex at $(h^j)^*$. 
Then for sufficiently large $i$, we can find $x^j_i\in E_i$, $j =1, ..., K$ with \begin{equation}\label{eq134}d_{GH}((B(x^j_i, \frac{1}{\epsilon}r_{x^j}), B_{W^j} (w_j, \frac{1}{\epsilon}r_{x^j}))<\epsilon r_{x^j};\end{equation} \begin{equation}\label{eq135}\cup_j B(x_i^j, \frac{\delta_5R}{2})\supset E_i.\end{equation}
Now we can apply the induction argument above for each geodesic ball $B(x_i^j, \frac{1}{\epsilon}r_{x^j})$. We obtain holomorphic functions $f_i^j$ in $B(q_i, \delta_3R)$ satisfying \begin{equation}\label{eq136}|f_i^j(x)|\geq\frac{1}{2}\end{equation} for $x\in B(x^j_i, \delta_5R)$; \begin{equation}\label{eq137}|f_i^j(x)|\leq C(n, v), x\in B(q_i, \frac{2\delta_3R}{3}); f_i^j(q_i) = 0.\end{equation}

 Put $G_i(x) = \sum\limits_{j=1}^K|f_i^j|^2+\sum\limits_{m=1}^{s-1}|w^i_m|^2$. Then by (\ref{eq100}), (\ref{eq103}) and (\ref{eq137}), \begin{equation}\label{eq138}|\nabla G_i(x)|\leq C(n, v), x\in B(q_i, \frac{\delta_3R}{2}); G_i(q_i) = 0;\end{equation}\begin{equation}\label{eq139} |G_i(x)|\geq \frac{1}{4}, x\in \partial B(q_i, \frac{\delta_3R}{3}).\end{equation} Therefore, there exists $\delta_6 = \delta_6(n, v)>0$ with \begin{equation}\label{eq140}\max\limits_{x\in B(q_i, \delta_6R)}|G_i(x)|\leq\frac{1}{10}.\end{equation} This contradicts the assumption that the proposition does not hold uniformly for $(Y_i^n, q_i)$. The proof of proposition \ref{prop3} is complete.
\end{proof}

We continue the proof of theorem \ref{thm4}. For any sequence $r_i\to \infty$, Set $(M_{i}, p_{i}) = (M, p, r_i^{-2}g)$(we shall make $r_i$ explicit in proposition \ref{prop4} below). Then there exist $R''_i\to\infty$ and metric cones $(X_{i}, x_i^*)$($x_i^*$ is the vertex) with \begin{equation}\label{eq141}d_{GH}(B(p_{i}, R''_i), B_{X_i}(x_i^*, R''_i))<\frac{1}{R''_i}.\end{equation} Let $d_i(x) = d_i(x, p_i)$ for $x\in M_i$.
Following the construction in (\ref{eq52}) and claim \ref{cl5}, we find a sequence $R'_i\to\infty$, functions $\rho_i$ in $M_i$ satisfying 
\begin{equation}\label{eq142}
\int_{B(p_i, 4R'_i)}|\nabla \rho_i-\nabla \frac{1}{2}d_i^2|^2 + |\nabla^2\rho_i-g_i|^2<\Phi(\frac{1}{i}).
\end{equation}
Also, in $B(p_i, 4R'_i)$, \begin{equation}\label{eq143}
|\rho_i-\frac{d_i^2}{2}|<\Phi(\frac{1}{i}); |\nabla \rho_i|\leq C(n)d_i.\end{equation}
As before, consider a smooth function $\varphi$: $\mathbb{R}^+\to \mathbb{R}^+$ with $\varphi(t) = t$ for $0\leq t\leq 1$; $\varphi(t) = 0$ for $t\geq 2$; $|\varphi|, |\varphi'|, |\varphi''|\leq C(n)$. Set \begin{equation}\label{eq144} v_i(z)=3(R'_i)^2\varphi(\frac{\rho_i(z)}{3(R'_i)^2})\end{equation} Then $v_i$ is supported in $B(p_i, 4R'_i)$. Let $H_i(x, y, t)$ be the heat kernel of $M_i$.
Consider the function $\tau_i(x) =\log (1+v_i(x))$ and define \begin{equation}\label{eq145}\tau_{i, t}(z) = \int_{M_i} H_i(z, y, t)\tau_i(y)dy.\end{equation}
Let $\delta_1 = \delta_1(n, v)$ be given by proposition \ref{prop3}. By (\ref{eq143}), we have \begin{equation}\label{eq146}\min\limits_{B(p_i, \frac{3}{2})\backslash B(p_i, \frac{30}{24})}\tau_i -\max\limits_{B(p_i, \frac{27}{24})}\tau_i\geq 2c(n, v)>0.\end{equation}
$\tau_i$ is of logarithmic growth uniform for all $i$. By heat kernel estimates, there exists $t_0 = t_0(n, v)>0$ so that \begin{equation}\label{eq147}\min\limits_{B(p_i, \frac{3}{2})\backslash B(p_i, \frac{30}{24})}\tau_{i, t_0} -\max\limits_{B(p_i, \frac{27}{24})}\tau_{i, t_0}\geq c(n, v)>0. \end{equation}

On a smooth K\"ahler metric cone, let $r$ be the distance function to the vertex. Then $\sqrt{-1}\partial\overline\partial\log(1+\frac{1}{2}r^2)$ is positive $(1, 1)$ form away from the vertex. Since $\tau_i$ resembles $\log(1+\frac{1}{2}r^2)$, by similar arguments as in proposition \ref{prop2}, we find that in $B(p_i, 5)$,
\begin{equation}\label{eq148}
\sqrt{-1}\partial\overline\partial\tau_{i, t_0}\geq c(n, v)>0.
\end{equation}
By proposition \ref{prop5}, for any fixed $R$ and sufficiently large $i$, in $B(p_i, R)$,
\begin{equation}\label{eq149}c(n, v)\log(d_i(x)+2)-C(n, v)\leq\tau_{i, t_0}(x)\leq C(n, v)\log(d_i(x)+2); \end{equation}
\begin{equation}\label{eq150}\sqrt{-1}\partial\overline\partial\tau_{i, t_0}(x)>0, x\in B(p_i, R).\end{equation}
Therefore, there exist sequences $\tilde{R}_i\to\infty, R_i\to\infty, c_i\to\infty$ so that $\tau^{-1}_{i, t_0}(\{c|c\leq c_i\})\cap B(p_i, \tilde{R}_i)$ is relatively compact on $B(p_i, \tilde{R}_i)$. Also \begin{equation}\label{eq151}\tau^{-1}_{i, t_0}(\{c|c\leq c_i\})\cap B(p_i, \tilde{R}_i)\supset B(p_i, R_i);\end{equation} \begin{equation}\label{eq152}\sqrt{-1}\partial\overline\partial\tau_{i, t_0}>0\end{equation} in $\tau^{-1}_{i, t_0}(\{c|c\leq c_i\})\cap B(p_i, \tilde{R}_i)$. Let $\Omega_i$ be the connected component of $\tau^{-1}_{i, t_0}(\{c|c< c_i\})$ containing $B(p_i, R_i)$. Then $\Omega_i$ is a Stein manifold.

According to proposition \ref{prop3}, there exist holomorphic functions $w_j^i(1\leq j\leq K=K(n, v))$ in $B(p_i, 3)$(here we take $R = \frac{3}{\delta_1}$ in proposition \ref{prop3}) so that \begin{equation}\label{eq153}w_j^i(p_i) = 0; \max_j\max\limits_{B(p_i, 1)}|w_j^i| = 1; \end{equation}
 \begin{equation}\label{eq154}\min\limits_{x\in\partial B(p_i, 1)}\sum\limits_{j=1}^K|w_j^i(x)|^2> 2\max\limits_{x\in B(p_i, \frac{3\delta_2}{\delta_1})}\sum\limits_{i=1}^K|w^i_j(x)|^2;\end{equation}  \begin{equation}\label{eq155}\frac{\max\limits_{x\in B(p_i, \frac{3}{2})}|w^i_j(x)|^2}{\max\limits_{x\in B(p_i, 1)}|w^i_j(x)|^2}\leq C(n, v).\end{equation}
 Then of course,  in $B(p_i, \frac{3}{2})$, \begin{equation}\label{eq156}|w^i_j(x)|\leq C(n, v).\end{equation}
 Also, by three circle theorem, we have \begin{equation}\label{eq157}\max_j\max\limits_{B(p_i, \frac{3\delta_2}{\delta_1})}|w_j^i|\geq c(n, v)>0.\end{equation}
 Thus \begin{equation}\label{eq158}\min\limits_{x\in\partial B(p_i, 1)}\sum\limits_{j=1}^K|w_j^i(x)|^2\geq c(n, v)>0. \end{equation}
 Now consider a cut off function $\lambda_i(x) =\lambda(d_i(x))$ with $\lambda_i = 1$ in $B(p_i, \frac{30}{24})$; $\lambda_i$ has compact support in $B(p_i, \frac{33}{24})$; $|\nabla\lambda_i|\leq C(n, v)$.
 Let $\tilde{w}^i_j = \lambda_iw^i_j$. Then $\overline\partial\tilde{w}^i_j$ is supported in $B(p_i, \frac{33}{24})\backslash B(p_i, \frac{30}{24})$. We solve the $\overline\partial$-problem $\overline\partial \tilde{f}^i_j = \overline\partial\tilde{w}^i_j$ in $\Omega_i$ with the weight function $\psi_i = \eta\tau_{i, t_0}$. Here $\eta = \eta(n, v)$ is a very large number to be determined. Then by (\ref{eq148}), we have \begin{equation}\label{eq159}\int_{\Omega_i}|\tilde{f}^i_j|^2e^{-\psi_i}\leq\frac{\int_{\Omega_i}|\overline\partial\tilde{w}^i_j|^2e^{-\psi_i}}{c(n, v)}. \end{equation} This implies that \begin{equation}\label{eq160}\int_{B(p_i, \frac{27}{24})}|\tilde{f}^i_j|^2e^{-\psi_i}\leq\frac{\int_{B(p_i, \frac{3}{2})\backslash B(p_i, \frac{30}{24})}|\overline\partial\tilde{w}^i_j|^2e^{-\psi_i}}{c(n, v)}.  \end{equation}
 Let \begin{equation}\label{eq161}f^i_j(x) = \tilde{w}^i_j(x) -\tilde{f}^i_j(x)-(\tilde{w}^i_j(p_i)-\tilde{f}^i_j(p_i)). \end{equation}
By (\ref{eq147}), (\ref{eq153}), (\ref{eq154}), (\ref{eq156}), (\ref{eq158}) and similar arguments as in (\ref{eq124}), if $\eta=\eta(n, v)$ is large enough, we can make $|\tilde{f}^i_j|$ so small in $B(p_i, 1)$ that 
\begin{equation}\label{eq162}
C(n, v)\geq\min\limits_{x\in\partial B(p_i, 1)}\sum\limits_{j=1}^K|f_j^i(x)|^2> \frac{3}{2}\max\limits_{x\in B(p_i, \frac{3\delta_2}{\delta_1})}\sum\limits_{i=1}^K|f^i_j(x)|^2\geq c(n, v)\end{equation}
Now we freeze the value $\eta = \eta(n, v)$. (\ref{eq149}) says $\psi_i$ is of logarithmic growth uniform for all $i$. By (\ref{eq159}) and the mean value inequality,  we find $C=C(n, v)>0$ so that for any $R>0$, if  $i$ is sufficiently large,
\begin{equation}\label{eq163}
|f_j^i(x)|\leq C(d_i(x)^{C}+1)
\end{equation}
for $x\in B(p_i, R)$.
By passing to subsequence, we can assume $(M_i, p_i)\to (M_\infty, p_\infty)$ in the Gromov-Hausdorff sense.
Also, $f^i_j$ converges to $f^\infty_j$ which is of polynomial growth of order $C$ on $M_\infty$.

\medskip

For $C$ in (\ref{eq163}), let $V = span\{g\in \mathcal{O}_{2C}(M)|g(p) = 0\}$ and let $k = dim(V)$. Take a basis $g_s$ of $V$ satisfying \begin{equation}\label{eq164}\int_{B(p, 1)}g_s\overline{g_t} = \delta_{st}.\end{equation}
\begin{prop}\label{prop4}
There exist constants $R>0$ and $c>0$ with $\sum\limits_{s}|g_s(x)|^2\geq cr(x, p)^2$ for $r(x, p)\geq R$.
\end{prop}
\begin{proof}
Assume the proposition is not true. There exist $r_i\to \infty$ and points $x_i$ with \begin{equation}\label{eq165}d(p, x_i) = r_i, \sum\limits_{s}|g_s(x_i)|^2\leq\frac{r_i^2}{i}.\end{equation}
We follow the notations from page $21$ to page $23$. For each $i$, There exists a basis $g^i_s$ of $V$ with \begin{equation}\label{eq166}\int_{B(p, 1)}g^i_s\overline{g^i_t} = \delta_{st}; \int_{B(p_i, 1)}g^i_s\overline{g^i_t}=\lambda^i_{st}\delta_{st}.\end{equation} Here $\lambda^i_{st}$ are constants. Then (\ref{eq164}) and (\ref{eq166}) imply \begin{equation}\label{eq167}\sum\limits_{s}|g_s|^2 = \sum\limits_{s}|g^i_s|^2. \end{equation}Note by three circle theorem and mean value inequality, \begin{equation}\label{eq168}\lambda^i_{ss}\geq cr_i^2\end{equation} for some $c= c(n, v)>0$. Then $h^i_s = \frac{g^i_s}{\sqrt{\lambda^i_{ss}}}$ satisfies \begin{equation}\label{eq169}\int_{B(p_i, 1)}h^i_s\overline{h^i_t} = \delta_{st}.\end{equation} Three circle theorem and mean value inequality imply \begin{equation}\label{eq170}0<c(n, v)\leq\max_{B(p_i, 1)}|h^i_s(x)|\leq C(n, v).\end{equation} After passing to subsequence, we may assume $M_i\to M_\infty$ and $h^i_s, f^i_j$ all converge. Say $h^i_s\to h^\infty_s$; $f^i_j\to f^\infty_j$ uniformly on each compact set. Clearly $h^\infty_s(s= 1,.., k)$ are linearly independent on $M_\infty$.
\begin{claim}\label{cl6}
$span\{f^\infty_j\}\subset span\{h^\infty_s\}$ on $M_\infty$.
\end{claim}
\begin{proof}
Assume the claim is not true. Set $V' =span\{f^\infty_j, h^\infty_s\}$. Then $dim(V')>k$. By three circle theorem, $f^\infty_j, h^\infty_s$ are of polynomial growth of order $2C$. Take a basis $u_1, ..., u_m$ of $V'$, $m\geq k+1$. Therefore, $u_l(1\leq l\leq m)$ are of polynomial growth of order $2C$.
For any $f\in V'$, $f$ satisfies three circle theorem. That is, if $M(f, r) = \max\limits_{B(p_\infty, r)}|f(x)|$, $\log M(f, r)$ is convex in terms of $\log r$. The reason is that $f$ is a limit of holomorphic functions of polynomial growth on $M_i$. Write $u_l = \sum\limits_{s=1}^ka^s_lh^\infty_s+\sum\limits_{j=1}^Kb^j_lf^\infty_j$. Here $a^s_l, b^j_l$ are constants. Define $u^i_l = \sum\limits_{s=1}^ka^s_lh^i_s+\sum\limits_{j=1}^Kb^j_lf^i_j$. Then $u^i_l\to u_l$ uniformly on each compact set. As $u_l$ is a basis for $V'$, for sufficiently large $i$, $u^i_l$ are linearly independent on $B(p_i, 1)$. We can also regard $u^i_l$ as holomorphic functions on $B(p, 3r_i)$ on $M$. Let $v^i_l$ be a basis of $span\{u^i_l\}$ with $\int_{B(p, 1)}v^i_l\overline{v^i_t} = \delta_{lt}$. Let us write $v^i_l = \sum\limits_{t=1}^mC^i_{lt}u^i_t$. Here $C^i_{lt}$ are constants. 
We are interested in \begin{equation}\label{eq171}F_{i, l} = \frac{\max\limits_{B(p_i, 2)}|v^i_l|}{\max\limits_{B(p_i, 1)}|v^i_l|}=\frac{\max\limits_{B(p_i, 2)}|\sum\limits_{t=1}^mC^i_{lt}u^i_t|}{\max\limits_{B(p_i, 1)}|\sum\limits_{t=1}^mC^i_{lt}u^i_t|}.\end{equation} In the quotient, we can normalize the coefficients $C^i_{lt}$ so that $\max\limits_{1\leq t\leq m}|C^i_{lt}| = 1$.
As $u_l$ are linearly independent on $M_\infty$, by a simple compactness argument and three circle theorem for $V'$ on $M_\infty$, we see that for $i$ sufficiently large, $1\leq l\leq m$, \begin{equation}\label{eq172}F_{i, l}\leq (2+\epsilon)^{2C}\end{equation} for any given $\epsilon>0$. As before, we can apply the three circle theorem to find a subsequence of $v^i_l$ converging to linearly independent holomorphic functions $v_l$ on $M$, satisfying $v_l(p) = 0$; $deg(v_l)\leq 2C$. As $l$ is from $1$ to $m$ and $m>k$, this contradicts that $dim(V) = k$.
\end{proof}
Given claim \ref{cl6}, we find $f^i_j$ is almost in the $span\{h^i_s\}$. More precisely, \begin{equation}\label{eq173}\lim\limits_{i\to\infty}\max_{B(p_i, 1)}|f^i_j(x) - \sum\limits_{s}c^i_{js}h^i_s| = 0\end{equation} for $c^i_{js} = \int_{B(p_i, 1)}f^i_j\overline{h^i_s}$. In particular, $|c^i_{js}|\leq C(n, v)$. By (\ref{eq162}), \begin{equation}\label{eq174}\min_{\partial B(p_i, 1)}\sum\limits_{j=1}^{K}|f^i_j(x)|^2 > \frac{3}{2}\max_{B(p_i, \frac{3\delta_2}{\delta_1})}\sum\limits_{j=1}^K|f^i_j(x)|^2\geq c(n, v)>0.\end{equation} Hence \begin{equation}\label{eq175}C(n, v)\min_{\partial B(p_i, 1)}\sum\limits_{s}|h^i_s|^2 \geq c(n, v)>0.\end{equation} Finally by (\ref{eq168}), \begin{equation}\label{eq176}|h^i_s|^2 = \frac{|g^i_s|^2}{\lambda^i_{ss}}\leq\frac{|g^i_s|^2}{cr_i^2}.\end{equation} Then from (\ref{eq167}), \begin{equation}\label{eq177}\min\limits_{\partial B(p, r_i)}\sum\limits_{s}|g_s|^2=\min\limits_{\partial B(p_i, 1)}\sum\limits_{s}|g_s|^2=\min\limits_{\partial B(p_i, 1)}\sum\limits_{s}|g^i_s|^2 \geq c(n, v)r_i^2>0.\end{equation} This contradicts (\ref{eq165}).

\end{proof}

To conclude the proof of theorem \ref{thm4}, we just need to add the constant function $1$ to $V$.

\end{proof}
\section{\bf{Completion of the proof of theorem \ref{thm1}}}
\begin{theorem}\label{thm12}
Let $M$ be a complete noncompact K\"ahler manifold with nonnegative bisectional curvature and maximal volume growth. Then $M$ is biholomorphic to an affine algebraic variety. Also the ring of holomorphic functions of polynomial growth is finitely generated.
\end{theorem}
\begin{proof}
Given any $k\in\mathbb{N}$, let $n_k=dim_{\mathbb{C}}(\mathcal{O}_k(M))$. Define a holomorphic map from $M$ to $\mathbb{C}^{n_k}$ by $F_k(x) = (g_1(x), .... , g_{n_k}(x))$. Here $g_1, ...., g_{n_k}$ is a basis for $\mathcal{O}_k(M)$. When $k$ is getting larger, we only add new functions to the basis (that is, we do not change the previous functions). Our goal is to prove that for sufficiently large $k$, $F_k$ is a biholomorphism to an affine algebraic variety. 

Below the value $k$ might change from line to line, basically we shall increase its value in finite steps.
First assume $k$ is large so that the functions $f_1, ... , f_N$ constructed in theorem \ref{thm4}  are in $\mathcal{O}_k(M)$ and they separate the tangent space at a point $p\in M$. Let $\alpha$ be the ideal of polynomial relations of functions $g_1, ..., g_{n_k}$. That is to say, \begin{equation}\label{eq178}\alpha = \{p(g_1, ..., g_{n_k})|p(g_1(x), ..., g_{n_k}(x)) = 0, \forall x\in M\}.\end{equation} Here $p$ is a polynomial. Then $\alpha$ is a prime ideal. Let $\Sigma_k$ be the affine algebraic variety defined by $\alpha$. Then $dim(\Sigma_k)= n$, as the transcendental dimension of $(g_1, ..., g_{n_k})$ over $\mathbb{C}$ is $n$. Moreover, $dim(F_k(M)) = n$, as the tangent space at $p$ is separated. 
By theorem \ref{thm4}, $F_k$ is a proper holomorphic map from $M$ to $\mathbb{C}^{n_k}$. Hence the image of $F_k$ is closed. By proper mapping theorem, the image of $F_k$ is an analytic subvariety of dimension $n$. As $\Sigma_k$ is irreducible, $F_k(M) = \Sigma_k$. 

Our argument below is very similar to some parts of \cite{[DS]}. Given any point in $\Sigma_k$, the preimage of $F_k$ is a compact subvariety of $M$, as $F_k$ is proper. As $M$ is a Stein manifold($M$ is exhausted by $\Omega_i$ which are Stein), The preimages contain only finitely many points. Given a generic point $y\in\Sigma_k$, we can find polynomial growth holomorphic functions separating $F_k^{-1}(y)$. Therefore, by increasing $k$, we may assume $F_k$ is generically one to one. Note that if $x\in\Sigma_k$ and the preimage of $x$ contain more than one point, then $x$ is in the singular set of $\Sigma_k$, say $S(\Sigma_k)$. Write $S(\Sigma_k)$ as a finite union of irreducible algebraic subvarieties $\Sigma'_s(1\leq s\leq t_k)$. Set $h = dim(S(\Sigma_k))$. Let us assume $dim(\Sigma'_s) = dim(S(\Sigma_k))$ for $1\leq s\leq r_k\leq t_k$. For a generic point $x\in\Sigma'_s$, the preimages under $F_k$ contain finitely many points. Therefore, we can increase the value of $k$ so that the preimages of $x$ and their tangent spaces are separated.  In this way, the dimension of $S(\Sigma_k)$ is decreased. After finitely many steps, $F_k$ becomes a biholomorphism from $M$ to $\Sigma_k$ which is affine algebraic.
\begin{claim}
We can identify polynomial growth holomorphic functions on $M$ with regular functions on $\Sigma_k$ via $F_k$. Thus $\mathcal{O}_P(M)$ is finitely generated.
\end{claim}
\begin{proof}
First, by theorem $3.2$ in \cite{[H]}, regular functions on $\Sigma_k$ are identified with the affine coordinate ring of $\Sigma_k$. Thus, any regular function is of polynomial growth. Since the transcendental dimension of $\mathcal{O}_P(M)$ is $n$ over $\mathbb{C}$, we may assume the affine coordinate functions generates the field of $\mathcal{O}_P(M)$. Then every polynomial growth holomorphic function is rational, hence a regular function on $M$.\end{proof}

The proof of theorem \ref{thm12} is complete.
\end{proof}
\begin{cor}
Let $M$ be a complete noncompact K\"ahler manifold with nonnegative bisectional curvature. Then the ring of holomorphic functions of polynomial growth is finitely generated.
\end{cor}
\begin{proof} 
We first consider the case when the universal cover does not split. 
By theorem $2$ in \cite{[L2]}, if there exists a nonconstant holomorphic function of polynomial growth on $M$, then $M$ is of maximal volume growth. Then then the result follows from the theorem above.

In the general case, let $\tilde{M}$ be the universal cover. Let $G$ be the fundamental group of $M$. 
Let $E$ be the set of $G$-invariant holomorphic functions of polynomial growth on $\tilde{M}$. We can identify $E$ with $\mathcal{O}_P(M)$. Given any $f\in E$, consider \begin{equation}\label{eq179}u^f_t(x) = \int_{\tilde{M}}H_{\tilde{M}}(x, y, t)\log(|f(y)|^2+1)dy,\end{equation} where $H_{\tilde{M}}(x, y, t)$ is the heat kernel of $\tilde{M}$. By theorem $3.1$ in \cite{[NT1]}, $\sqrt{-1}\partial\overline\partial u^f_t\geq 0$ for $t>0$. Let $D^t_f$ be the null space of $\sqrt{-1}\partial\overline\partial u^f_t$.  Theorem $3.1$ in \cite{[NT1]} says $D^t_f$ is a parallel distribution.  

\begin{claim}$D^t_f$ is invariant for $t>0$. Then we define $D_f = D^t_f$, $t>0$.
\end{claim}\label{cl8}
\begin{proof} By theorem $2.1$, part ($ii$) in \cite{[NT1]}(see also the second sentence in the proof of corollary $2.1$ in \cite{[NT1]}), if $t_1>t_2>0$, 
\begin{equation}\label{eq180}dim(D^{t_1}_f)\leq dim(D^{t_2}_f).\end{equation} De Rham theorem says we can write $\tilde{M} = N_1\times N_2$ where $D^{t_2}_f$ is the tangent space of $N_2$. $u^f_{t_2}$ is of logarithmic growth by proposition \ref{prop5}. Moreover, $u^f_{t_2}$ is pluriharmonic on each slice $N_2$, hence harmonic on $N_2$. As $N_2$ has nonnegative bisectional curvature, the Ricci curvature of $N_2$ is nonnegative. By a theorem of Cheng-Yau \cite{[CY]}, $u^f_{t_2}$ is constant on each slice $N_2$. That is to say, $u^f_{t_2}$ is a function on $N_1$. By uniqueness of the heat flow, $u^f_{t_1}$ is also constant on each slice of $N_2$. Combining this with (\ref{eq180}), we obtain that $D^{t_1}_f=D^{t_2}_f$. 
\end{proof}

Hence, $u^f_t$ is constant on $N_2$ for $t\geq 0$. This implies $f$ is constant on the factor $N_2$.
Now define the parallel distribution \begin{equation}\label{eq181}D = \cap_{f\in E}D_f.\end{equation} By De Rham theorem, we can assume $\tilde{M} = M_1\times M_2$ where $D$ is the tangent space of $M_2$. Then, for any $f\in E$, $f$ is constant on the factor $M_2$. Note $D$ is invariant under $G$-action. 
Fix an inclusion $i$ of a slice: $M_1\hookrightarrow M_1\times M_2$. Now for any $g\in G$, $g(i(M_1))$ must be another slice of $M_1$. 
Let $\pi$ be the projection from $M_1\times M_2$ to $M_1$.  For $x\in M_1$ and $g\in G$, Define a holomorphic isometry $u_g$ of $M_1$ by $u_g(x)=\pi(g(i(x)))$. Of course, $u_g$ is a subgroup of the holomorphic isometry group of $M_1$. Let $G'$ be the closure of $u_g$. Then we can identify $E$ with polynomial growth holomorphic functions on $M_1$ invariant under $G'$.
\begin{claim}\label{cl9}
$G'$ is a compact group.
\end{claim}
\begin{proof}
It suffices to prove that for $x\in M_1$, $u_g(x)$ is bounded for $g\in G$. Assume this is not true, then there exists a sequence $g_i\in G'$ with $x_i=g_i(x)\to\infty$ on $M_1$.
Let $(U, z_1, .., z_m)$ be a holomorphic chart on $M_1$ around $x$ with $z(x) = 0$. Let $(U_i = g_i(U), z^i_s = z_s\circ g^{-1}_i)$ be the holomorphic chart in $U_i$. By taking subsequence of $x_i$, we may assume $U_i$ are disjoint.
We will use some construction in \cite{[N1]}. First, pick finitely many $f_j\in E$ so that $\sqrt{-1}\partial\overline\partial\sum\limits_ju^{f_j}_1>0$ on $M_1$. Let $u = \sum\limits_ju^{f_j}_1$.
Then $u$ is a strictly plurisubharmonic function on $M_1$ with logarithmic growth. Moreover, $u$ is invariant under $G'$ action. Let $U^2\subset\subset U^1\subset\subset U$ be open sets containing $x$. Consider a smooth cut-off function $\varphi$ with $\varphi = 1$ in $U^2$; $\varphi = 0$ in $M_1\backslash U^1$. Define $\varphi_i = \varphi\circ g_i^{-1}$. Then $\varphi_i$ is supported in $U_i$. Let \begin{equation}\label{eq182}\psi(x) = 4m\sum\limits_{i}\varphi_i(x)\log(\sum\limits_{s=1}^m|z^i_s(x)|^2) + Cu(x).\end{equation} Here $C$ is a positive constant so that $\sqrt{-1}\partial\overline\partial\psi\geq\omega$ on $U_i$; $\omega$ is the K\"ahler form on $M_1$. Then $\sqrt{-1}\partial\overline\partial\psi > 0$ on $M_1$.
Now we solve the $\overline\partial$-problem $\overline\partial h_i = \overline\partial\varphi_i$ with \begin{equation}\label{eq183}\int_{M_1}|h_i|^2e^{-\psi}\leq\int_{M_1}|\overline\partial\varphi_i|^2e^{-\psi}.\end{equation} One sees that $\lambda_i = h_i-\varphi_i$ are holomorphic functions of polynomial growth. The growth orders are uniformly bounded. Moreover, $h_i(x_k) = 0$ for all $k\in\mathbb{N}$. Thus $\lambda_i$ are linearly independent, as $(h_i-\varphi_i)(x_j) = -\delta_{ij}$. This contradicts theorem \ref{thm9}.
\end{proof}
\begin{claim}\label{cl10}
$M_1$ is of maximal volume growth. In particular, the ring of polynomial growth holomorphic functions is finitely generated.
\end{claim}
\begin{proof}
As $M_1$ is simply connected, write $M_1$ as a product of K\"ahler manifolds which are not products anymore. For each factor, there exists polynomial growth holomorphic function on $M_1$ which is not constant on that factor. Then each factor must be of maximal volume growth by theorem $2$ in \cite{[L2]}.
\end{proof}

By claim \ref{cl10} and theorem \ref{thm12}, $\mathcal{O}_P(M_1)$ is finitely generated. $\mathcal{O}_P(M)$ is just the subring of $\mathcal{O}_P(M_1)$ invariant under $G'$. 
Since $G'$ is compact, the finite generation of $\mathcal{O}_P(M)$ follows from a theorem of Nagata \cite{[Na]} (the detailed argument is in the appendix $B$).

\end{proof}
\appendix
\section{Proof of theorem \ref{thm6}}
\begin{proof}
This part basically follows from \cite{[NT1]}. For any $a>0$, $\eta(x, t)$ satisfies (\ref{eq184}), (\ref{eq185}) below.
\begin{equation}\label{eq184}
(\frac{\partial}{\partial t}-\Delta)\eta_{\gamma\overline\delta} = R_{\beta\overline\alpha\gamma\overline\delta}\eta_{\alpha\overline\beta}-\frac{1}{2}(R_{\gamma\overline{p}}\eta_{p\overline\delta}+R_{p\overline\delta}\eta_{\gamma\overline{p}})
\end{equation}
\begin{equation}\label{eq185}
\int_M||\eta(x, 0)||\exp(-ar^2(x))dx<\infty
\end{equation}
We will assume the equation below at this moment. The proof is given at the end of this section.
\begin{equation}\label{eq186}
\lim\limits_{r\to\infty}\inf\int_0^T\int_{B(p, r)}||\eta||^2(x, t)\exp(-ar^2(x))dxdt<\infty
\end{equation}

Recall corollary $1.1$ in \cite{[NT1]} with simplifications the assumptions:
\begin{prop}\label{prop5}
Let $(M^n, p)$ be a complete noncompact K\"ahler manifold with nonnegative bisectional curvature. $r(x) = d(x, p)$. Let $u$ be a nonnegative function on $M$ satisfying \begin{equation}\label{eq187}\dashint_{B(p, r)}u(y)dy \leq \exp(a+br)\end{equation} for some constants $a, b>0$. Let \begin{equation}\label{eq188}v(x, t) = \int_MH(x, y, t)u(y)dy.\end{equation} $H$ is the heat kernel on $M$. Then given any $\epsilon>0, T>0$, there exists $C(n, \epsilon, a, b)>0$  such that for any $x$ satisfying $r=r(x)\geq\sqrt{T}$, \begin{equation}\label{eq189}-C(n, \epsilon, a, b)+C_1(n, \epsilon)\inf_{B(x, \epsilon r)}u\leq v(x, t)\leq C(n, \epsilon, a, b)+\sup_{B(x, \epsilon r)}u\end{equation} for $0\leq t\leq T$. Here $C_1(n, \epsilon)>0$.
\end{prop}
Fix a point $p\in M$. Let $r(x) = d(x, p)$. Let $\phi(x) = \exp(r(x))$. Define \begin{equation}\label{eq190}\phi(x, t) = e^t\int_MH(x, y, t)\phi(y)dy.\end{equation} Then \begin{equation}\label{eq191}(\frac{\partial}{\partial t}-\Delta)\phi = \phi\end{equation} and \begin{equation}\label{eq192}\phi(x, t)\geq ce^{c_1r}\end{equation} for $0\leq t\leq T$, by proposition \ref{prop5}. Here $c, c_1$ are positive constants. Let \begin{equation}\label{eq193}h(x, t) = \int_MH(x, y, t)||\eta||(y, 0)dy.\end{equation}
The proposition below is just lemma $2.2$ in \cite{[NT1]}.
\begin{prop}\label{prop6}
There exists a positive function $\tau(R)$ so that for $0\leq t\leq T$, $h(x, t)\leq \tau(R)$ for $x\in B(p, 2R)\backslash B(p, \frac{R}{2})$. Moreover, $\lim\limits_{R\to\infty}\tau(R) = 0$.
\end{prop}
The next proposition is lemma 2.1 in \cite{[NT1]}. Note (\ref{eq184}), (\ref{eq185}) and (\ref{eq186}) are used.
\begin{prop}\label{prop7}
$||\eta||(x, t)$ is a subsolution of the heat equation.  Moreover, $||\eta||(x, t)\leq h(x, t)$.
\end{prop}

Given $\epsilon>0$, define \begin{equation}\label{eq194}(\tilde\eta)_{\alpha\overline\beta} = \eta_{\alpha\overline\beta}+(\epsilon\phi-\lambda(x, t)) g_{\alpha\overline\beta}.\end{equation} At $t=0$, $\tilde\eta> 0$. Also, for $0\leq t\leq T$, if $R$ is sufficiently large, by proposition \ref{prop6}, we have $\tilde\eta>0$ on $\partial B(p, R)$. Suppose at some $t_0\in [0, T]$, $\tilde\eta(x_0, t_0)<0$ for $x_0\in \overline{B(p, R)}$. Then there exists $0\leq t_1<T$ with $\tilde\eta(x, t)\geq 0$ for $x\in B(p, R)$ and $0\leq t\leq t_1$. Moreover, the minimum eigenvalue of $\tilde\eta(x_1, t_1)$ is zero for some $x_1\in B(p, R)$(note $x_1$ cannot be on the boundary).  Now we apply the maximal principle. Let us assume \begin{equation}\label{eq195}\tilde\eta(x_1, t_1)_{\gamma\overline{\gamma}}=0\end{equation} for $\gamma\in T^{1,0}_{x_1}M, |\gamma| =1$. We may diagonize $\tilde\eta$ at $(x_1, t_1)$. Of course, we can assume $\gamma$ is one of the basis of the holomorphic tangent space. Then at $(x_1, t_1)$,
 \begin{equation}\label{eq196}(\frac{\partial}{\partial t}-\Delta)\tilde\eta_{\gamma\overline{\gamma}}\leq 0.\end{equation} On the other hand, by (\ref{eq184}), \begin{equation}\label{eq197}\begin{aligned}(\frac{\partial}{\partial t}-\Delta)\eta_{\gamma\overline{\gamma}} &= \sum\limits_{\alpha}R_{\gamma\overline{\gamma}\alpha\overline\alpha}\eta_{\alpha\overline\alpha}-\sum\limits_{\alpha}R_{\gamma\overline{\gamma}\alpha\overline\alpha}\eta_{\gamma\overline{\gamma}}\\&=\sum\limits_{\alpha}R_{\gamma\overline{\gamma}\alpha\overline\alpha}(\tilde\eta_{\alpha\overline\alpha}-\tilde\eta_{\gamma\overline\gamma})\\&\geq 0.\end{aligned}\end{equation} Note \begin{equation}\label{eq198}(\frac{\partial}{\partial t}-\Delta)((\epsilon\phi-\lambda(x, t)) g_{\gamma\overline{\gamma}}) = \epsilon\phi g_{\gamma\overline{\gamma}}> 0.\end{equation} Hence at $(x_1, t_1)$, \begin{equation}\label{eq199}(\frac{\partial}{\partial t}-\Delta)\tilde\eta_{\gamma\overline{\gamma}}>0.\end{equation} This is a contradiction. Now let $R\to\infty$ and then $\epsilon\to 0$, we proved that $\eta-\lambda(x, t)g_{\alpha\overline\beta}\geq 0$ for $0\leq t\leq T$.

Next we verify (\ref{eq186}). Basically we follow page $487-488$ on \cite{[NT1]}. Note our condition is more special.
First, we have that $|v(x, t)|\leq C$ for all $x, t$, as $u$ has compact support. Note \begin{equation}\label{eq200}(\Delta-\frac{\partial}{\partial t})v^2 = 2|\nabla v|^2.\end{equation}  Multiplying (\ref{eq200}) by suitable cutoff functions, using integration by parts, we find \begin{equation}\label{eq201}\int_0^T\dashint_{B(p, r)}|\nabla v|^2\leq C_1(r^{-2}\int_0^{2T}\dashint_{B(p, 2r)}v^2+\dashint_{B(p, r)}u^2)\leq C_2(T+1)\end{equation} for $r\geq 1$.
Bochner formula gives \begin{equation}\label{eq202}(\Delta-\frac{\partial}{\partial t})|\nabla v|^2\geq 2|\nabla^2 v|^2.\end{equation} Multiplying (\ref{eq202}) by suitable cutoff functions, using integration by parts, we find \begin{equation}\label{eq203}\int_0^T\dashint_{B(p, r)}|\nabla^2v|^2\leq C_3(r^{-2}\int_0^{2T}\dashint_{B(p, 2r)}|\nabla v|^2+\dashint_{B(p, r)}|\nabla u|^2)\leq C_2(T+1)\end{equation} for $r\geq 1$.
From this, (\ref{eq186}) follows easily.

\end{proof}

\section{Some algebraic results of Nagata}
We continue the proof of theorem \ref{thm1}.
The ring $R=\mathcal{O}_P(M_1)$ is finitely generated. We may assume the generators are in $F=\mathcal{O}_d(M_1)$ for some $d>0$. Let $g_1, ..., g_m$ be a basis for $F$.  Obviously $F$ is an invariant space of $G'$. Then we may think $\mathcal{O}_P(M_1)$ is $\mathbb{C}[g_1, ..., g_m]\slash\alpha$. Here $\alpha$ is an ideal. Then the $G'$ action on $R$ is induced by the representation $G'\to GL(m, \mathbb{C})$. Let $I_{G'}(R)$ be the subring of $R$ fixed by $G'$. 
In \cite{[Na]}, page $370$, the following definition was made:
\begin{definition}
A group $G$ is reductive if every rational representation is completely reducible.
\end{definition}
It was pointed out on page $370$ of \cite{[Na]} that all rational representations of $G$ in \cite{[Na]} are given by some specific finite dimensional representations of $G$.
In our case, as $G'$ is compact, every finite dimensional representation(complex) is completely reducible. Therefore, according to the definition above, $G'$ is reductive.
In \cite{[Na]}, the following was proved:
\begin{theorem}[Nagata]
$I_G(R)$ is finitely generated if $G$ is semi-reductive. 
\end{theorem}
It was pointed out in the first sentence of part $5$, page $373$ of \cite{[Na]} that a reductive group is obviously semi-reductive.
Putting all these things together, we proved the finite generation of $I_{G'}(R) = \mathcal{O}_P(M)$.

\end{document}